\documentclass{amsart}

\usepackage{tikz}
\usepackage{amsmath}
\usepackage{amsthm}
\usepackage{amssymb}
\usepackage{graphicx}
\usepackage{subfigure}
\usepackage{hyperref}
\usepackage[lite]{amsrefs}
\usepackage[margin=2.5cm]{geometry}
\usepackage{float}
\usepackage{scalerel}
\usetikzlibrary{calc,decorations.pathreplacing}

\allowdisplaybreaks[1]

\numberwithin{equation}{section}

\theoremstyle{plain}
\newtheorem{thm}[equation]{Theorem}
\newtheorem{cor}[equation]{Corollary}
\newtheorem{lem}[equation]{Lemma}
\newtheorem{prop}[equation]{Proposition}
\newtheorem{conj}[equation]{Conjecture}

\theoremstyle{definition}
\newtheorem{defn}[equation]{Definition}

\theoremstyle{remark}
\newtheorem{rem}[equation]{Remark}
\newtheorem{ex}[equation]{Example}

\DeclareRobustCommand\C{\mathbb{C}}
\newcommand{\cc}{\mathcal{C}}

\renewcommand{\S}{\mathcal{S}}

\newcommand{\T}{\mathbb{T}}
\newcommand{\Z}{\mathbb{Z}}

\newcommand{\F}{\mathbb{F}}

\newcommand{\tr}{^\mathrm{T}}
\newcommand{\g}{\mathfrak{g}}
\newcommand{\inv}{^{-1}}
\newcommand{\diag}{\mathrm{diag}}

\newcommand{\xx}{\mathbf{x}}
\newcommand{\slt}{\mathfrak{sl}_2}

\newcommand{\cqg}{\C_q[G]}
\newcommand{\prim}{\mathrm{Prim}\,}
\newcommand{\rr}{\mathrm{r}}
\newcommand{\supp}{\mathrm{Supp}}
\newcommand{\uqg}{U_q(\mathfrak{g})}
\newcommand{\spec}{\mathrm{Spec}\,}
\newcommand{\ww}{w_1,w_2}
\newcommand{\trs}{\mathcal{T}}

\newcommand{\I}{\mathcal{I}}
\newcommand{\ii}{\mathbf{i}}
\newcommand{\Span}{\mathrm{Span}}
\newcommand{\wtw}{W\times W}
\newcommand{\cqt}{\C_q[SL_2]}
\newcommand{\N}{\mathbb{Z}_{\ge0}}
\newcommand{\cqb}[1]{\mathbb{C}_{q_{#1}}[G_{#1}]}
\newcommand{\cqbp}[1]{\mathbb{C}_{q_{#1}}[B^+_{#1}]}
\newcommand{\cqbm}[1]{\mathbb{C}_{q_{#1}}[B^-_{#1}]}

\newcommand{\iim}{(i_1,i_2,\cdots,i_m)}

\newcommand{\np}[2]{c_{#1,#2}^+}

\newcommand{\nm}[2]{c_{#1,#2}^-}

\newcommand{\ee}{\mathbf{e}}
\newcommand{\sgn}{\mathrm{sgn}}
\newcommand{\tm}{\tilde{{M}}}
\newcommand{\tpi}{\tilde{\pi}}
\renewcommand{\aa}{\mathbf{a}}
\newcommand{\lqh}{L_q[H]}

\newcommand{\lqt}[1]{L_{q_{#1}}(2)}
\newcommand{\muu}{\boldsymbol{\mu}}
\newcommand{\pp}{\mathcal{P}}
\newcommand{\tildeflag}[1]{#1}
\newcommand{\ti}{\tildeflag{i}}
\newcommand{\tii}{\tildeflag{\ii}}
\newcommand{\tiii}[1]{(\ti_1,\ti_2,\cdots,\ti_{#1})}
\newcommand{\tiim}{(\ti_1,\ti_2,\cdots,\ti_m)}
\newcommand{\nn}{\mathbf{n}}
\newcommand{\nnm}{(n_1,n_2,\cdots,n_m)}
\newcommand{\zz}{\mathcal{Z}}
\renewcommand{\wp}[1]{w_{>#1}^+}
\newcommand{\wm}[1]{w_{>#1}^-}
\newcommand{\nl}[2]{c_{#1,#2}^l}
\newcommand{\nr}[2]{c_{#1,#2}^r}
\newcommand{\bb}{\mathbf{b}}

\newcommand{\nuu}{\boldsymbol{\nu}}

\renewcommand{\th}{\tilde{H}}
\newcommand{\hh}{\hat{H}}
\newcommand{\ta}{\tilde{A}}
\newcommand{\tb}{\tilde{B}}
\newcommand{\tc}{\tilde{C}}
\newcommand{\tp}{\tilde{P}}
\newcommand{\jj}{\mathbf{j}}
\newcommand{\te}{\tilde{e}}
\newcommand{\tee}{\tilde{\mathbf{e}}}

\newcommand{\mm}{\mathbf{m}}

\newcommand{\irr}{\mathrm{Irr}\,}
\newcommand{\tirr}{\widetilde{\mathrm{Irr}}\,}
\newcommand{\omegaa}{\boldsymbol{\omega}}
\newcommand{\tPhi}{\tilde{\Phi}}
\newcommand{\tOmega}{\tilde{{\Omega}}}
\newcommand{\id}{\mathrm{Id}}

\renewcommand{\tt}{\mathbf{t}}
\newcommand{\ds}{\mathbf{s}}
\newcommand{\rrr}{\mathbf{r}}
\newcommand{\uu}{\mathbf{u}}
\newcommand{\epss}{\boldsymbol{\varepsilon}}
\newcommand{\ttii}{\trs_{\tii}}
\newcommand{\ttiit}{\trs_{\tii}^\times}
\newcommand{\tmii}{\tm_{\tii}}
\newcommand{\ttiz}{\trs_{\tii}^0}
\newcommand{\ttizt}{(\trs_{\tii}^0)^\times}
\newcommand{\cM}{\mathcal{M}}
\newcommand{\taa}{\tilde{\mathbf{a}}}
\newcommand{\tla}{\tilde{a}}
\newcommand{\ccc}{\mathbf{c}}

\newcommand{\lqi}{L_q(\tii)}
\newcommand{\specm}{\mathrm{Specm}\,}
\newcommand{\zew}{\Z_{E_{\ww}^{\pm1}}}
\newcommand{\ppp}{\mathbf{p}}
\newcommand{\tpp}{\tilde{\mathbf{p}}}

\newcommand{\boxx}[2]{
	\ifnum #2>0
		\draw (#1-1,#2-1)--(#1,#2) (#1-1,#2-1)--(#1,#2-1) (#1-1,#2)--(#1,#2);
		\node at (#1-1,#2-1) [circle,fill,inner sep=1pt] {};
		\node at (#1,#2) [circle,fill,inner sep=1pt] {};
	\else
		\draw (#1-1,-#2)--(#1,-#2-1) (#1-1,-#2-1)--(#1,-#2-1) (#1-1,-#2)--(#1,-#2);
		\node at (#1,-#2-1) [circle,fill,inner sep=1pt] {};
		\node at (#1-1,-#2) [circle,fill,inner sep=1pt] {};
	\fi
}
\newcommand{\boxl}[2]{
	\boxx{#1}{#2}
	\ifnum #2>0
		\node at (#1-0.5,#2-0.85) {$\scriptstyle x$};
		\node at (#1-0.5,#2+0.15) {$\scriptstyle x^{\scaleto{-1}{3pt}}$};
		\node at (#1-0.5,#2-0.35) {$\scriptstyle y$};
	\else
		\node at (#1-0.5,-#2-0.85) {$\scriptstyle x$};
		\node at (#1-0.5,-#2+0.15) {$\scriptstyle x^{\scaleto{-1}{3pt}}$};
		\node at (#1-0.5,-#2-0.35) {$\scriptstyle y$};
	\fi
}
\newcommand{\boxs}[1]{
	\setcounter{cnt}{0}
	\foreach \s in #1 {
		\stepcounter{cnt}
		\boxx{\value{cnt}}{\s}
	}
}
\newcommand{\boxsl}[1]{
	\setcounter{cnt}{0}
	\foreach \s in #1 {
		\stepcounter{cnt}
		\boxl{\value{cnt}}{\s}
	}
}
\newcommand{\lines}[2]{
	\foreach \s in {1,...,#1} {
		\draw (0,\s-1)--(#2,\s-1);
		\foreach \ss in {0,...,#2} {
			\node at (\ss,\s-1) [circle,fill,inner sep=1pt] {};
		}
		\node at (-0.2,\s-1) {$\s$};
		\node at (#2+0.2,\s-1) {$\s$};
	}
}
\newcommand{\linesl}[2]{
	\setcounter{cnt}{0}
	\foreach \s in #2 {
		\stepcounter{cnt}
		\setcounter{cnt1}{-1}
		\foreach \ss in {1,...,#1} {
			\ifnum \ss<\s
				\node at (\value{cnt}-0.5,\ss-0.85) {$\scriptstyle 1$};
			\else
				\ifnum \ss<-\s
					\node at (\value{cnt}-0.5,\ss-0.85) {$\scriptstyle 1$};
				\fi
			\fi
			\stepcounter{cnt1}
			\ifnum \value{cnt1}>\s
				\ifnum \value{cnt1}>-\s
					\node at (\value{cnt}-0.5,\ss-0.85) {$\scriptstyle 1$};
				\fi
			\fi
		}
	}
}

\newcommand{\boxdiagram}[2]{
	\boxs{#2}
	\lines{#1}{\value{cnt}}
}
\newcommand{\boxdiagraml}[2]{
	\boxsl{#2}
	\lines{#1}{\value{cnt}}
	\linesl{#1}{#2}
}
\newcommand{\drawpath}[1]{
	\setcounter{cnt1}{0}
	\foreach \s in #1 {
		\stepcounter{cnt1}
		\setcounter{cnt2}{-1}
		\foreach \ss in #1 {
			\stepcounter{cnt2}
			\ifnum \value{cnt1}=\value{cnt2}
				\draw[red,ultra thick,->] (\value{cnt1}-1,\s-1)--(\value{cnt2},\ss-1);
			\fi
		}
	}
}

\newcounter{cnt}
\newcounter{cnt1}
\newcounter{cnt2}

\begin{document}

	\title{Representations of Quantum Coordinate Algebras at Generic $q$ and Wiring Diagrams}

	\author[He Zhang]{He Zhang}
	\address{Department of Mathematical Sciences, Tsinghua University, Beijing,  China}
	\email{he-zhang17@mails.tsinghua.edu.cn}
	\author[Hechun Zhang]{Hechun Zhang}
	\thanks{Hechun Zhang was partially supported by National Natural Science Foundation of China grants No. 12031007 and No. 11971255}
	\address{Department of Mathematical Sciences, Tsinghua University, Beijing,  China}
	\email{hczhang@tsinghua.edu.cn}
	\author[R. B. Zhang]{Ruibin Zhang}
	\thanks{Ruibin Zhang was partially supported by Australian Research Council grant DP170104318.}
	\address{School of Mathematics and Statistics,
	The University of Sydney, Sydney, NSW 2006,  Australia}
	\email{ruibin.zhang@sydney.edu.au}
	
	\begin{abstract}
		This paper is devoted to the representation theory of quantum coordinate algebra $\cqg$, for a semisimple Lie group $G$ and a generic parameter $q$. By inspecting the actions of normal elements on tensor modules, we generalize a result of Levendorski\u{\i} and Soibelman in \cite{LS} for highest weight modules. For a double Bruhat cell $G^{\ww}$, we describe the primitive spectra $\prim\cqg_{\ww}$ in a new fashion, and construct a bundle of $(\ww)$ type simple modules onto $\prim\cqg_{\ww}$, provided $\supp(w_1)\cap\supp(w_2)=\varnothing$ or enough pivot elements. The fibers of the bundle are shown to be products of the spectrums of simple modules of 2-dimensional quantum torus $L_q(2)$. As an application of our theory, we deduce an equivalent condition for the tensor module to be simple, and construct some simple modules for each primitive ideal when $G=SL_3(\C)$. This completes the Dixmier's program for $\C_q[SL_3]$. The wiring diagrams, introduced by Fomin and Zelevinsky in their study of total positivity (cf. \cite{BFZ,FZ}), is the main tool to compute the action of generalized quantum minors on tensor modules in the type A case. We obtain a quantum version of Lindstr\"{o}m's lemma, which plays an important role in transforming representation problems into combinatorial ones of wiring diagrams.
	\end{abstract}

	\maketitle
	\tableofcontents
	
	\section{Introduction}

	Let $G$ be a finite dimensional semisimple Lie group over $\C$ with Lie algebra $\g$. The quantized universal enveloping algebra $\uqg$ and the \emph{quantum coordinate algebra} $\cqg$ are dual Hopf algebras in some appropriate sense. While the representation theory of $\uqg$ is in many aspects standard, the representation theory of $\cqg$ is to a large extent unknown. However, the latter one is more suitable for the consideration of deformation of $G$, following the general philosophy of studying a geometric object through the sheaf of functions on it. Hence, the representation theory of $\cqg$ is worth studying in more detail (c.f. \cite{Jo,LS,DP,Zh}).

	Throughout the paper, we assume the parameter $q\in\C^*$ is generic, i.e. $q$ is not a root of 1, unless otherwise stated.

	The initial goal of studying the representation theory of $\cqg$ is to classify the simple $\cqg$-modules, which unfortunately is unattainable (as for most algebras, c.f. \cite{Bl}). We instead follow Dixmier's proposal to classify the primitive ideals (the annihilators of simple modules), and construct at least one simple module for each primitive ideal.

	The first step of the Dixmier's program for $\cqg$ was achieved by Joseph \cite{Jo,Jo1}, who proved that the primitive ideals of $\cqg$ are in bijection with the symplectic leaves of $G$ as Poisson-Lie group (also see \cite{HLT} for a generalization of Joseph's results to the multi-parameters quantum groups). Furthermore, Yakimov \cite{Ya} proved that the bijection is actually equivariant with respect to some torus action. Goodearl and Letzter \cite{GL} proved $\cqg$ satisfies the \emph{Dixmier-Moeglin equivalence}.

	Let $W$ be the Weyl group of $G$. According to Joseph (Theorem \ref{thm-stratification}), the prime spectrum of $\cqg$ has the decomposition
    \[\spec\cqg=\bigsqcup_{(\ww)\in \wtw}\spec\cqg_{\ww},\]
	where $\cqg_{\ww}$ is the quantum coordinate algebra corresponding to the double Bruhat cell $G^{\ww}=B^+w_1B^+\cap B^-w_2B^-$. Furthermore, $\spec\cqg_{\ww}$ is homeomorphic to the prime spectrum of its center $Z_{\ww}=Z(\cqg_{\ww})$, and the primitive spectrum $\prim\cqg_{\ww}$ is homeomorphic to the maximal spectrum of $Z_{\ww}$, via extension and contraction.

	However, the second step of the Dixmier's program for $\cqg$ is largely open; only a small class of irreducible representations are known. Levendorsky and Soibelman \cite{LS} constructed a class of simple $\cqg$-modules corresponding to the double Bruhat cells $G^{w,w}$, which we call \emph{Soibelman modules}. Recently, Saito \cite{Sa}, Tanisaki \cite{Ta} gave some new constructions of Soibelman modules and related them to the PBW bases of $\uqg$. Roughly speaking, the Soibelman modules are the simple highest weight $\cqg$-modules with a properly chosen triangular decomposition of $\cqg$. We also call Soibelman modules Fock representations as the underlying spaces are polynomial algebras and $\cqg$ acts by some difference operators. In \cite{Na}, Narayanan generalized Soibelman's result to arbitrary Kac-Moody algebra with symmetrizable generalized Cartan matrix.

	The Dixmier's program for $\cqt$ lays a foundation for the general cases. The primitive ideals of $\cqt$, characterized by a pair of Weyl group elements $(e,e),(s,e),(e,s)$ or $(s,s)$, are given explicitly in (\ref{eq-sltprim}) (c.f. \cite{HLT}). The type $(e,e)$ simple $\cqt$-modules are $1$-dimensional and thus of little interest. The simple $\cqt$-modules of the other types, under some semisimple conditions, are determined completely in Section \ref{sec-TypMod}. As vector spaces, they are isomorphic to either the polynomial algebra or the Laurent polynomial algebra in one variable. We refer to them as \emph{typical $\cqt$-modules}. By Theorem \ref{prop-typmod}, they are the only simple $\cqt$-modules with semisimple actions of some maximal commutative subalgebra.

	For general $G$, there is a natural and straightforward construction of $\cqg$-modules from tensor products of lifted $\cqt$-modules, following the construction of Soibelman modules. Namely, corresponding to each simple root $\alpha_i$ of $G$, there is a restriction map $\pi_i:\cqg\twoheadrightarrow\C_{q_i}[SL_2]_i$, where $\C_{q_i}[SL_2]_i$ is the quantum coordinate algebra of the subgroup $SL_2\subset G$ associated to $\alpha_i$. The maps $\pi_i$'s enable one to lift any $\C_{q_i}[SL_2]_i$-modules to $\cqg$-modules. We show that for any $w\in W$ with a reduced expression $(i_1,i_2,\cdots,i_m)\in\I_w$, the tensor product of $\pi_{i_k}$-lifted $(s,s)$-typical $\C_{q_{i_k}}[SL_2]_{i_k}$-modules (in the order specified by the reduced expression) leads to a simple $\cqg$-module of $(w,w)$ type.

	The proof of this result boils down to a detailed analysis of the actions of some special elements of $\cqg$, called \emph{(quasi-)normal elements}, and makes essential use of structural properties of Demazure modules. The Soibelman modules constructed in \cite{LS} are the special case with all tensor factors being $(s,s)$-typical \emph{highest weight modules} (see Section \ref{sec-TypMod}).

	To construct simple modules of $(\ww)$ type for arbitrary $(\ww)\in\wtw$, we proceed by tensoring lifted $(s,e),(e,s)$-typical modules according to some double reduced expression $\tii=\tiim\in\I_{\ww}$. Here in $\tii=\tiim$ we use negative indices for simple reflections in the first copy of $W$, and positive indices in the second copy. By Corollary \ref{cor-w1w2type}, we obtain a tensor module $\tmii$ of $(\ww)$ type in the sense of \cite{HLT}.
	
	However, $\tmii$ may not be simple, thus we need to analyze its structure in more detail. Let $m=l(w_1)+l(w_2)$ be the length of $\tii$. Denote $L_q(2)$ the 2-dimensional quantum torus algebra with generators $x^{\pm1},y^{\pm1}$ subject to $xy=qyx$. The tensor module $\tmii$ then admits an $\lqi$-module structure, where $\lqi=\lqt{|\ti_1|}\otimes\lqt{|\ti_2|}\otimes\cdots\otimes\lqt{|\ti_m|}$, and $\cqg$ acts on $\tmii$ via some map $\tpi_{\tii}:\cqg\to \lqi$. We introduce the notion of \emph{weight string}, which is a tuple of $m$ weights of specific form depending on $\tii$ (Definition \ref{defn-wtstr}). For each weight string $\muu$ we associate a unit $I(\muu)$ in $\lqi$. A calculation shows images of matrix coefficients under $\tpi_{\tii}$ are linear combinations of $I(\muu)$'s (Proposition \ref{prop-pii}). Let $R_{\tii}$ be the image of $\tpi_{\tii}$, $S_{\tii}$ be the intersection of $R_{\tii}$ with $\lqi^\times$ (set of units in $\lqi$). One of our main theorems (Theorem \ref{thm-localization}) shows $\ttii:=R_{\tii}[S_{\tii}\inv]$ is a quantum torus embedded in $\lqi$, generated by $q$-commute elements $I(\muu)^{\pm1}$, where $\muu$ runs through the weight strings of type $\tii$. The structure of $\ttii$ is studied explicitly. In particular, the center $Z(\ttii)$ of $\ttii$ is shown to be $Z_{\ww}$, so the primitive spectra $\prim\cqg_{\ww}$ can be characterized as the maximal spectrum of $Z(\ttii)$ (Proposition \ref{prop-prim}). The dimension of $Z(\ttii)$ is calculated to be $\dim\ker(w_1-w_2)$ using the root data of $\tii$ (see Section \ref{sec-appen1}), which agrees with the result by De Concini and Procesi in \cite{DP}.
	
	The tensor module $\tmii$ as $\ttii$-module admits abundant maximal $\ttii$-submodules. By Theorem \ref{thm-simquot}, we are able to construct a bundle $\tirr\ttii$ of simple quotient $\ttii$-modules of $\tmii$ onto $\prim\cqg_{\ww}$, with fibers (see (\ref{eq-tirrti}))
	\[\irr L_{q^{m_{\tii,1}}}(2)\times\irr L_{q^{m_{\tii,2}}}(2)\times\cdots\times\irr L_{q^{m_{\tii,k(\tii)}}}(2)\]
	for some $m_{\tii,1},m_{\tii,2},\cdots,m_{\tii,k(\tii)}\in\N$, where $\irr\lqt{}$ stands for the spectrum of simple modules of $\lqt{}$, $k(\tii)=(l(w_1)+l(w_2)-|\supp(\ww)|-\dim\ker(w_1-w_2))/2$.
	
	Some additional conditions are needed to show $R_{\tii}$-submodules of $\tmii$ are in bijection with $\ttii$-submodules. Namely, the existence of enough \emph{pivot elements} in $R_{\tii}$ (Theorem \ref{prop-quanlind}). The idea of pivot elements is to find some elements in $R_{\tii}$ whose action on $\tmii$ is ``invertible'' in $R_{\tii}$. A special case of enough pivot elements is when $\supp(w_1)\cap\supp(w_2)=\varnothing$ (Proposition \ref{prop-expivot}). If these conditions are satisfied, we can obtain the same bundle $\tirr\ttii$ of $(\ww)$ type simple $\cqg$-modules onto $\prim\cqg_{\ww}$. 
	
	As a direct application of the theory above, we deduce that the tensor module $\tmii$ is simple if and only if $|\supp(\ww)|=l(w_1)+l(w_2)$ (Theorem \ref{thm-irrtenmod}). Besides, our theory suggests a possibility of realizing the Dixmier's program for $\cqg$ by finding enough pivot elements in $R_{\tii}$. In the type A case, this can be done in a combinatorial way. We use the wiring diagram, which was introduced by Fomin and Zelevinsky in their study of total positivity (cf. \cite{BFZ,FZ}), to compute actions of quantum minors on tensor modules $\tmii$. The main calculation tool is a quantum version of Lindstr\"{o}m's lemma, which shows the image of quantum minors under $\tpi$ is the sum of weights of families of paths (Proposition \ref{prop-quanlind}). In the $G=SL_3(\C)$ case, we list the construction of enough pivot elements in Table \ref{tab-constructions}, which realizes the Dixmier's program in this case.

	We point out that the study of $\cqg$'s representation theory has gone much further when $q$ is a root of 1, which is of more geometric flavor. De Concini and Procesi in \cite{DP} establish a geometric theory of representations of $\cqg$ when $q$ is an $l$-th primitive root of 1. The main theorem states that simple modules of $(\ww)$ type form a principal bundle onto the double Bruhat cell $G^{\ww}$, and gives their Galois group (\cite{DP}*{Theorem 4.10}). Besides, the simple modules of $(\ww)$ type all have dimension $l^{\frac{1}{2}\dim\mathcal{O}_{\ww}}$, where $\mathcal{O}_{\ww}$ is one of the symplectic leaves contained in $G^{\ww}$ (all conjugate under the action of maxiaml torus) and
	\[\dim\mathcal{O}_{\ww}=l(w_1)+l(w_2)+\rr(w_1-w_2).\]
	The authors then describe the irreducibility of the tensor module using dimensions of the symplectic leaves: given two simple modules of $(\ww)$ type and of $(w_1',w_2')$ type such that $l(w_1w_1')=l(w_1)+l(w_1')$ and $l(w_2w_2')=l(w_2)+l(w_2')$, the tensor module of the two is simple if and only if
	\[\dim\mathcal{O}_{w_1w_1',w_2w_2'}=\dim\mathcal{O}_{\ww}+\dim\mathcal{O}_{w_1',w_2'}.\]
	When $q$ is generic, we notice the Soibelman modules satisfies the above relation of dimensions of symplectic leaves automatically. Besides, the condition in Theorem \ref{thm-irrtenmod} is also equivalent to the relation above. This indicates some similarities in the representation theories of $\cqg$ no matter $q$ is a root of 1 or not.

	\section{Preliminaries}

	\subsection{Quantized Universal Enveloping Algebra}

	We recall some basic concepts and notations from quantized universal enveloping algebras, i.e. quantum groups of Drinfeld-Jimbo type. Let $\g$ be a semisimple Lie algebra of rank $n$. Denote by $\Pi=\{\alpha_1,\alpha_2,\cdots,\alpha_n\}$ the set of simple roots. Let $(\cdot,\cdot)$ be the inner product of the Euclidean space spanned by $\alpha_1,\alpha_2,\cdots,\alpha_n$. The quantized universal enveloping algebra $\uqg$ is defined as the $\C$-algebra with generators $E_i,F_i,K_i^\pm,1\le i\le n$ and relations
	\begin{equation}
		\begin{gathered}\label{eq-DJ}
			K_iK_i\inv=1=K_i\inv K_i,K_iK_j=K_jK_i,\\
			K_iE_jK_i\inv=q_i^{c_{ij}}E_j,\\
			K_iF_jK_i\inv=q_i^{-c_{ij}}F_j,\\
			E_iF_j-F_jE_i=\delta_{ij}\frac{K_i-K_i\inv}{q_i-q_i\inv},\\
			\sum_{k=0}^{1-c_{ij}}(-1)^kE_i^{(1-c_{ij}-k)}E_jE_i^{(k)}=0,i\ne j,\\
			\sum_{k=0}^{1-c_{ij}}(-1)^kF_i^{(1-c_{ij}-k)}F_jF_i^{(k)}=0,i\ne j,
		\end{gathered}	
	\end{equation}
	where $c_{ij}=(\alpha_i^\vee,\alpha_j),q_i=q^{(\alpha_i,\alpha_i)/2}$. As usual, the $q$-integers, $q$-factorials, and $q$-binomial coefficients are denoted by
	\[[k]_q=\frac{q^k-q^{-k}}{q-q\inv},[k]_q!=[1]_q[2]_q\cdots[k]_q,\begin{bmatrix}m\\k\end{bmatrix}_q=\frac{[m]_q!}{[k]_q![m-k]_q!},k,m\in\N,k\le m,\]
	and the divided power is denoted by
	\[E_i^{(k)}=\frac{E_i^k}{[k]_{q_i}!},F_i^{(k)}=\frac{F_i^k}{[k]_{q_i}!}.\]
	The Hopf algebra structure of $\uqg$ is given by
	\[\Delta(K_i)=K_i\otimes K_i,\varepsilon(K_i)=1,S(K_i)=K_i\inv,\]
	\[\Delta(E_i)=E_i\otimes1+K_i\otimes E_i,\varepsilon(E_i)=0,S(E_i)=-K_i\inv E_i,\]
	\[\Delta(F_i)=1\otimes F_i+F_i\otimes K_i\inv,\varepsilon(F_i)=0,S(F_i)=-F_iK_i.\]

	Denote by $W$ the Weyl group. $W$ is generated by simple reflections $s_i,1\le i\le n$. Denote by $l(w)$ the length of $w\in W$, and $w_0$ the longest element in $W$. Given $a,b\in\Z\cup\{\pm\infty\},a\le b$, let $[a,b]$ be the closed interval of integers between $a,b$. Denote by $\I_w$ be the set of reduced expressions of $w\in W$, i.e.
	\[\I_w=\{\ii=(i_1,i_2,\cdots,i_{l(w)})\in[1,n]^{l(w)}|w=s_{i_1}s_{i_2}\cdots s_{i_{l(w)}}\}.\]
	Let $\I_{\ww}$ be the set of reduced expressions of $(\ww)\in \wtw$. To avoid confusion, we will use the indices $-1,-2\cdots,-n$ for the simple reflections in the first copy of $W$, and $1,2\cdots,n$ for the second copy. In this way $\I_{\ww}$ can be written as
	\[\I_{\ww}=\{\tii=\tiii{l(w_1)+l(w_2)}\in([-n,n]\backslash\{0\})^{l(w_1)+l(w_2)}\big|w_1=\prod_{\ti_k<0}s_{-\ti_k},w_2=\prod_{\ti_k>0}s_{\ti_k}\}.\]
	For $\ii\in\I_w$ or $\tii\in\I_{\ww}$, denote by $l(\ii)=l(w)$ or $l(\tii)=l(w_1)+l(w_2)$ the length of $\ii$. Denote by $\supp(\ww)$ the set $\{|\ti_1|,|\ti_2|,\cdots,|\ti_{l(\ii)}|\}$ for some $\tii\in\I_{\ww}$. Define $\supp(w)$ to be $\supp(e,w)$. One can check that the definition of $\supp(\ww)$ is independent of the choice of $\tii\in\I_{\ww}$. 

	From now on we abbreviate $\uqg$ to $U$. For $1\le i\le n$, define $U(i)=\langle E_i,K_i^\pm\rangle$, $U(-i)=\langle F_i,K_i^\pm\rangle$ and $U_i=\langle E_i,F_i,K_i^\pm\rangle$. For $\ii=\iim\in\I_w$ define $U_{\ii}=U_{i_1}U_{i_2}\cdots U_{i_m}$. For $\tii=\tiim\in\I_{\ww}$ define $U(\tii)=U(\ti_1)U(\ti_2)\cdots U(\ti_m)$. Denote by $U^+=\langle E_i\rangle_{i=1}^n$, $U^-=\langle F_i\rangle_{i=1}^n$, and $H=\langle K_i^\pm\rangle_{i=1}^n$.

	\subsection{Braid Actions and Demazure Modules}

	Denote by $P$ and $P^+$ the set of integral weights and dominant integral weights respectively. Let $\omega_1,\omega_2,\cdots,\omega_n$ be the fundamental weights. Denote by $\cc$ the category of finitely dimensional type 1 $U$-modules. Recall that a simple module in $\cc$ is parameterized by its highest weight $\lambda$ in $P^+$, and denote it by $V(\lambda)$. All modules in $\cc$ are direct sums of $V(\lambda)$'s.

	Let $M(\lambda)$ be the Vermma module of $U$ with highest weight $\lambda\in P^+$ and highest weight vector $u_\lambda$. $M(\lambda)$ is a free $U^-$-module. $V_\lambda$ is realized as a quotient module of $M(\lambda)$ through
	\begin{equation}\label{eq-hwm}
		V(\lambda)\simeq M(\lambda)/\big(\sum_{i=1}^nUF_i^{(\lambda,\alpha_{i}^\vee)+1}u_\lambda\big).
	\end{equation}

	For some $V\in\cc$ and $\mu\in P$ define $V_\mu=\{v\in V|K_iv=q^{(\alpha_i,\mu)}v,1\le i\le n\}$ to be the weight space of $\mu$. The braid group action on $V$ are characterized by Lusztig's symmetries $T_i,1\le i\le n$. Namely, for $v\in V_\mu$,
	\begin{equation}\label{eq-Ti}
		T_i(v)=\sum_{-a+b-c=(\mu,\alpha_i^\vee)}(-1)^bq_i^{b-ac}E_i^{(a)}F_i^{(b)}E_i^{(c)}v,
	\end{equation}
	cf. \cite{Ja}*{Chapter 8}. For any reduced expression $w=s_{i_1}s_{i_2}\cdots s_{i_r}$ of $w\in W$, define $T_w$ to be $T_{i_1}T_{i_2}\cdots T_{i_r}$. It is well known that $T_w$ is independent of the reduced expression chosen. The actions of $T_w$ on $U$ and on $V$ are compatible in the sense that $T_w(x.v)=T_w(x).T_w(v),x\in U,v\in V$.

	The $T_w$'s action on $V$ satisfies $T_w(V_\mu)=V_{w(\mu)}$ for $\mu\in P$. This implies $\dim V(\lambda)_{w(\lambda)}=1$ for any $w\in W$. For each $V(\lambda),\lambda\in P^+$, choose some highest weight vector $v_\lambda\in V(\lambda)_\lambda$. Define $v_{w(\lambda)}:=T_w(v_\lambda)\in V(\lambda)_{w(\lambda)}$, choose $v^*_{w(\lambda)}\in V(\lambda)^*$ such that $v^*_{w(\lambda)}(v_{w(\lambda)})=1$, and $v^*_{w(\lambda)}$ vanishes on any other weight spaces of $V(\lambda)$.

	The following result concerns $T_w$'s action on tensor modules, due to \cite{KR,LS1,Ta}.

	\begin{prop}\label{prop-DeltaTw}
		Choose some $w\in W$ and $\ii=\iim\in\I_w$, then
		\[\Delta(T_w)=(T_w\otimes T_w)\exp_{q_{i_1}}(X_1)\exp_{q_{i_2}}(X_2)\cdots\exp_{q_{i_m}}(X_m),\]
		where
		\[\exp_q(x)=\sum_{k=0}^\infty\frac{q^{k(k-1)/2}}{[k]_q!}x^k,\]
		\[X_k=(q_{i_k}-q_{i_k}\inv)\,T_{i_m}\inv\cdots T_{i_{k+1}}\inv(F_{i_k})\otimes T_{i_m}\inv\cdots T_{i_{k+1}}\inv(E_{i_k}).\]
	\end{prop}

	As a corollary, we want to mention the following property which we will use later.

	\begin{cor}\label{cor-Twaction}
		Let $\lambda_1,\lambda_2\in P^+$, then $T_w(v_{\lambda_1}\otimes v_{\lambda_2})=T_w(v_{\lambda_1})\otimes T_w(v_{\lambda_2})$.
	\end{cor}

	For $\lambda\in P^+$ and $w\in W$, define the \emph{Demazure modules} $V_{w,\lambda}^\pm$ as
	\[V_{w,\lambda}^+=U^+V(\lambda)_{w(\lambda)}\subset V(\lambda),\]
	\[V_{w,\lambda}^-=U^-V(-w_0(\lambda))_{-w(\lambda)}\subset V(-w_0(\lambda)).\]
	We will frequently look into the structure of the Demazure modules, which is demonstrated in the following proposition.

	\begin{prop}[\cite{DP}*{Proposition 2.2}]\label{prop-DemMod}
		Let $\lambda\in P^+,w\in W$ and $\ii=\iim\in\I_w$. Then $V_{w,\lambda}^+$ is spanned by elements $F_{i_1}^{h_1}F_{i_2}^{h_2}\cdots F_{i_m}^{h_m}v_\lambda,h_k\ge0,1\le k\le m$. Similar result holds for $V_{w,\lambda}^-$.
	\end{prop}

	\begin{cor}\label{cor-DemMod}
		(i) Let $\lambda\in P^+,w\in W$ and $\ii\in\I_w$. Then $V_{w,\lambda}^+=U_{\ii}v_\lambda$, $V_{w,\lambda}^-=U_{\ii}v_{-\lambda}$.

		(ii) Let $\lambda\in P^+,(\ww)\in\wtw$ and $\tii\in\I_{\ww}$. Then $V_{w_1,\lambda}^+=U(\tii)v_\lambda$, $V_{w_2,\lambda}^-=U(\tii)v_{-\lambda}$.
	\end{cor}

	\subsection{Quantum Coordinate Algebra}
	
	Let $G$ be some semisimple Lie group over $\C$ with Lie algebra $\g$. By definition the \emph{quantum coordinate algebra} $\C_q[G]$ is the subspace of $U^*$ spanned by matrix coefficients $c_{f,v}^V$, where $V\in\cc,f\in V^*,v\in V$, and
	\[c^V_{f,v}(u)=f(u.v),\forall u\in U.\]
	If $V^*$ is given a $U$-module structure by $(xf)(v)=f(S(x)v),x\in U,f\in V^*,v\in V$, then $\C_q[G]$ is endowed with a dual Hopf algebra structure
	\begin{equation}\label{eq-cqghopf}
		\begin{gathered}
			c^{V_1}_{f_1,v_1}+c^{V_2}_{f_2,v_2}=c^{V_1\oplus V_2}_{f_1\oplus f_2,v_1\oplus v_2},c^{V_1}_{f_1,v_1}c^{V_2}_{f_2,v_2}=c^{V_1\otimes V_2}_{f_1\otimes f_2,v_1\otimes v_2},\\
			\Delta(c^V_{f,v})=\sum_i c^V_{f,\mathbf{v}_i}\otimes c^V_{\mathbf{f}_i,v},\varepsilon(c^V_{f,v})=f(v),S(c^V_{f,v})=c^{V^*}_{v,f}.
		\end{gathered}
	\end{equation}
	where $V_i\in\cc,f_i\in V_i^*,v_i\in V_i,\{\mathbf{f}_i\},\{\mathbf{v}_i\}$ is a pair of dual basis for $V^*$ and $V$.

	$\cqg$ has a left $U\otimes U$-module structure defined by $(x_1\otimes x_2)c_{f,v}^V=c_{x_1f,x_2v}^V,x_1,x_2\in U$. The following is a well-known theorem concerning structure of $\cqg$.

	\begin{thm}[Peter-Weyl]\label{thm-PW}
		The $U\otimes U$-homomorphisms $V(\lambda)^*\otimes V(\lambda)\to\cqg$ for $\lambda\in P^+$, defined by taking matrix coefficients, sum up to give a $U\otimes U$-isomorphism
		\[\bigoplus_{\lambda\in P^+}V(\lambda)^*\otimes V(\lambda)\to\cqg\]
	\end{thm}

	The theory of $R$-matrix implies the following commutative relations for matrix coefficients:

	\begin{prop}[\cite{LS,DP}]\label{prop-RmatComRel}
		Let $V_1,V_2\in\cc,v_1\in (V_1)_{\mu_1},v_2\in (V_2)_{\mu_2},f_1\in (V_1^*)_{\nu_1},f_2\in (V_2^*)_{\nu_2}$. Then
		\[c_{f_1,v_1}^{V_1}c_{f_2,v_2}^{V_2}=q^{(\nu_1,\nu_2)-(\mu_1,\mu_2)}c_{f_2,v_2}^{V_2}c_{f_1,v_1}^{V_1}+\sum_i c_{f_{2i},v_{2i}}^{V_2}c_{f_{1i},v_{1i}}^{V_1},\]
		where
		\[f_{2i}\otimes f_{1i}=p_i(q)(M_i(E)\otimes M_i(F))(f_2\otimes f_1),v_{2i}\otimes v_{1i}=p_i'(q)(M_i'(E)\otimes M_i'(F))(v_2\otimes v_1),\]
		here $p_i(q),p_i'(q)\in\C(q)$ and $M_i,M_i'$ are monomials of which at least one is not constant.
	\end{prop}

	\begin{rem}\label{rem-quanmat}
		One can deduce generators and relations of $\cqg$ for specific $G$. When $G=SL_{n+1}$, it has Weyl group $W\simeq S_{n+1}$ generated by simple reflections $s_i=(i,i+1),1\le i\le n$. Take simple roots as $\alpha_i=\varepsilon_i-\varepsilon_{i+1}$, where $\varepsilon_1,\varepsilon_2,\cdots,\varepsilon_{n+1}$ form a orthonormal basis in some Euclidean space. $V(\omega_1)$ is the so-called \emph{natural representation} of $U$, and every $V(\omega_i)$ appears as a direct summand of some tensor power of $V(\omega_1)$.  Write $V(\omega_1)=\C v_1\oplus\C v_2\oplus\cdots\oplus\C v_{n+1}$, where $v_1$ is some highest vector and $v_{i+1}=F_iv_i$. We can choose generators of $\C_q[SL_{n+1}]$ to be $x_{ij}=c^{V(\omega_1)}_{v_i^*,v_j},1\le i,j\le n+1$, where $v_i^*(v_j)=\delta_{ij}$, and the relations are
		\begin{equation}\label{eq-quanmat}
			\begin{gathered}
				x_{ij}x_{il}=qx_{il}x_{ij},j<l,\\
				x_{ij}x_{kj}=qx_{kj}x_{ij},i<k,\\
				x_{il}x_{kj}=x_{kj}x_{il},i<k,j<l,\\
				x_{ij}x_{kl}-x_{kl}x_{ij}=(q-q^{-1})x_{il}x_{kj},i<k,j<l,\\
				\mathrm{det}_q(x_{ij})=1,
			\end{gathered}
		\end{equation}
		where $\det_q$ stands for the quantum determinant: for indeterminates $y_{ij},1\le i,j\le k$, define
		\[\mathrm{det}_q(y_{ij}):=\sum_{\tau\in S_{k}}(-q)^{l(\tau)}y_{1\tau(1)}y_{2\tau(2)}\cdots y_{k\tau(k)}.\]
	\end{rem}

	\subsection{Spectra of $\cqg$}

	In this subsection we review some results on spectra of $\cqg$ in Joseph \cite{Jo} and Yakimov \cite{Ya}. Define the subalgebras of $\cqg$
	\[R^+=\Span\{c_{f,v_\lambda}^{V(\lambda)}|\lambda\in P^+,f\in V(\lambda)^*\},\]
	\[R^-=\Span\{c_{f,v_{w_0(\lambda)}}^{V(\lambda)}|\lambda\in P^+,f\in V(\lambda)^*\}.\]
	Joseph proved that
	\begin{equation}\label{eq-R+R-}
		\cqg=R^+R^-=R^-R^+.
	\end{equation}
	Define
	\[J_w^+=\Span\{c_{f,v_\lambda}^{V(\lambda)}|\lambda\in P^+,f\in(V_{w,\lambda}^+)^\perp\}\subset R^+,w\in W,\]
	\[J_w^-=\Span\{c_{f,v_{-\lambda}}^{V(-w_0(\lambda))}|\lambda\in P^+,f\in(V_{w,\lambda}^-)^\perp\}\subset R^-,w\in W,\]
	\[J_{\ww}=J_{w_1}^+R^-+R^+J_{w_2}^-\subset\cqg,(\ww)\in \wtw.\]
	By \emph{$H$-prime}, we mean an $H$-invariant completely prime ideal of $R^\pm$ or $\cqg$ with respect to the $1\otimes H$-action. In fact, $H$-primes are finite in number and can be characterized in the following manner (first due to Joseph, \cite{Jo}*{Proposition 10.1.8, Proposition 10.3.5}).
	\begin{thm}[Yakimov, \cite{Ya}*{Theorem 2.1}]\label{thm-Hprime}
		(i) For each $w\in W$, $J_w^\pm$ is an $H$-prime of $R^\pm$. All $H$-primes of $R^\pm$ are of this form.

		(ii) For each $(\ww)\in \wtw$, $J_{\ww}$ is an $H$-prime of $\cqg$. All $H$-primes of $\cqg$ are of this form.
	\end{thm}

	Define \emph{normal elements} of $\cqg$ to be
	\[\np{w}{\lambda}=c_{v_{w(\lambda)}^*,v_\lambda}^{V(\lambda)},\nm{w}{\lambda}=c_{v_{-w(\lambda)}^*,v_{-\lambda}}^{V(-w_0(\lambda))},\]
	for $\lambda\in P^+,w\in W$. By Corollary \ref{cor-Twaction}, one has
	\begin{equation}\label{eq-NmElt}
		\np{w}{\lambda_1}\np{w}{\lambda_2}=\np{w}{\lambda_1+\lambda_2},\nm{w}{\lambda_1}\nm{w}{\lambda_2}=\nm{w}{\lambda_1+\lambda_2},\lambda_1,\lambda_2\in P^+,w\in W.
	\end{equation}
	So for any $\mu\in P,\mu=\lambda_1-\lambda_2,\lambda_1,\lambda_2\in P^+$, one can define $c_{w,\mu}^\pm=c_{w,\lambda_1}^\pm(c_{w,\lambda_2}^\pm)\inv$, which is independent of the choice of $\lambda_1,\lambda_2\in P^+$ by (\ref{eq-NmElt}). Noticing Equation (\ref{eq-R+R-}) and Theorem \ref{thm-Hprime}, the $R$-matrix commutative relations in Proposition \ref{prop-RmatComRel} for $\np{w}{\lambda},\nm{w}{\lambda}$ can be rewritten as:
	\begin{prop}\label{prop-NormComRel}
		Let $w\in W,\lambda\in P^+,V\in\cc,f\in (V^*)_{-\nu},v\in V_\mu,\nu,\mu\in P$. Then
		\[c_{f,v}^V\np{w}{\lambda}=q^{(w(\lambda),\nu)-(\lambda,\mu)}\np{w}{\lambda}c_{f,v}^V\;\;\mathrm{mod}\,J_w^+R^-,\]
		\[c_{f,v}^V\nm{w}{\lambda}=q^{(w(\lambda),\nu)-(\lambda,\mu)}\nm{w}{\lambda}c_{f,v}^V\;\;\mathrm{mod}\,R^+J_w^-.\]
	\end{prop}
	
	Let $w\in W$ and $(\ww)\in \wtw$. By abuse of notation we may denote the images of $c_{w,\lambda}^\pm$ in $R^\pm/J_w^\pm$ or $\cqg/J_{\ww}$ by the same symbols. Denote the multiplicative subset of $R^\pm, R^\pm/J_w^\pm$ or $\cqg/J_{\ww}$
	\[E_w^\pm=\{c_{w,\lambda}^\pm|\lambda\in P^+\}.\]
	Denote the multiplicative subset of $\cqg$ or $\cqg/J_{\ww}$
	\[E_{\ww}=E_{w_1}^+E_{w_2}^-.\]
	Proposition \ref{prop-NormComRel} suggests that $E_{\ww}\subset\cqg/J_{\ww}$ satisfies the left and right Ore conditions. Thus one may define the localization
	\[\cqg_{\ww}=(\cqg/J_{\ww})[E_{\ww}\inv]\]
	and its center
	\[Z_{\ww}=Z(\cqg_{\ww}).\]

	Before stating Joseph's Stratification Theorem, which is first due to Joseph \cite{Jo,Jo1}, recall the torus $\T^n=(\C^*)^n$ acts on $\cqg$ by
	\begin{equation}\label{eq-torus-action}
		t\cdot c_{f,v}^V=\prod_{i=1}^nt_i^{(\mu,\alpha_i^\vee)}\,c_{f,v}^V,\ t=(t_1,t_2,\cdots,t_n)\in\T^n,V\in\cc,f\in V^*,v\in V_\mu.
	\end{equation}
	This extends the $1\otimes H$-action on $\cqg$.

	\begin{thm}[Yakimov, \cite{Ya}*{Theorem 2.3}]\label{thm-stratification}
		(i) For each prime ideal $J$ of $\cqg$, there exists unique $(\ww)\in \wtw$ such that $J\supset J_{\ww}$ and $J\cap E_{\ww}=\varnothing$, i.e.
		\[\spec\cqg=\bigsqcup_{(\ww)\in \wtw}\spec\cqg_{\ww}.\]
		(ii) For each $(\ww)\in\wtw$, $Z_{\ww}$ is isomorphic to a Laurent polynomial ring of dimension $\dim\ker(w_1-w_2)$. Denote by $i_{\ww}:Z_{\ww}\to\cqg_{\ww}$ the natural embedding. Then $i_{\ww}^*$ gives a homeomorphism from $\spec\cqg_{\ww}$ to $\spec Z_{\ww}$.

		(iii) For each $(\ww)\in\wtw$, the primitive spectrum of $\cqg_{\ww}$
		\[\prim\cqg_{\ww}=(i_{\ww}^*)\inv(\specm Z_{\ww}),\]
		and $\prim\cqg_{\ww}$ is a single $\T^n$-orbit.
	\end{thm}

	\begin{thm}[Yakimov, \cite{Ya}*{Theorem 3.1}]\label{thm-genzww}
		Let $\lambda^{(1)},\lambda^{(2)},\cdots,\lambda^{(p(\ww))}$ be the $\Z$-basis of
		\[\ker(w_1-w_2)\bigcap\bigoplus_{i\in\supp(\ww)}\Z\omega_i\subset P,\]
		where $p(\ww)=\dim\ker(w_1-w_2)+|\supp(\ww)|-n$, then $Z_{\ww}$ is generated by
		\[\np{w_1,\omega_i},i\notin\supp(\ww);\np{w_1}{\lambda^{(j)}}(\nm{w_2}{\lambda^{(j)}})\inv,1\le j\le p(\ww),\]
		and their inverses in $E_{\ww}E_{\ww}\inv$ as Laurent polynomial algebra.
	\end{thm}

	\subsection{Quantum Torus}\label{sec-qt}

	\sloppy A \emph{quantum torus of dimension $k$} is by definition an associative algebra with generators $x_1^{\pm1},x_2^{\pm1},\cdots,x_k^{\pm1}$ subject to relations $x_ix_j=q^{h_{ij}}x_jx_i$, where $q\in\C^*$ is a parameter, and $H=(h_{ij})$ is an $k\times k$ skew-symmetric integer matrix. We denote the algebra by $\lqh$.
	
	We summarize some results on the structure of $\lqh$ below.

	\begin{prop}\label{prop-qtiso}
		For two $k\times k$ skew symmetric interger matrix $H,H'$, $\lqh\simeq L_q[H']$ as algebra if there exits some $P\in GL_k(\Z)$, such that $P\tr HP=H'$.
	\end{prop}

	For $\aa=(a_1,a_2,\cdots,a_k)\tr\in\Z^k$, write $x^{\aa}$ for the monomial $x_1^{a_1}x_2^{a_2}\cdots x_k^{a_k}$. Then $\{x^{\aa}|\aa\in\Z^k\}$ forms basis of $\lqh$. For any $k\times k$ skew-symmetric integer matrix $H$, there exists some $Q\in GL_k(\Z)$ such that
	\[Q\tr HQ=\mathrm{diag}(m_1S,m_2S,\cdots,m_lS,0,\cdots,0),S=\begin{pmatrix}0&1\\-1&0\end{pmatrix},\]
	where $m_1\le m_2\le\cdots\le m_l\in\N$ and $2l$ equals the rank $\rr(H)$ of $H$. For simplicity, we will denote $L_q[S]$ by $L_q(2)$, with generators $x^{\pm1},y^{\pm1}$ subject to $xy=qyx$. The following is a direct application of Proposition \ref{prop-qtiso}.

	\begin{cor}\label{cor-qtcenter}
		Keep notations above, then
		\[\lqh\simeq L_{q^{m_1}}(2)\otimes L_{q^{m_2}}(2)\otimes\cdots\otimes L_{q^{m_l}}(2)\otimes Z,\]
		where $Z$ is the center of $\lqh$, which is a Laurent polynomial algebra in $k-\rr(H)$ variables with a basis $\{x^{\aa}|H\aa=0,\aa\in\Z^k\}$.
	\end{cor}

	Corollary \ref{cor-qtcenter} suggests that simple modules of $\lqh$ can be obtained from those of $\lqt{}$'s. Denote $\irr\lqt{}$ the spectrum of simple modules of $\lqt{}$. Write $\lqh$ in the form of Corollary \ref{cor-qtcenter}, one then constructs
	\begin{equation}\label{eq-tirr}
		\tirr\lqh=\irr L_{q^{m_1}}(2)\times\irr L_{q^{m_2}}(2)\times\cdots\times\irr L_{q^{m_l}}(2)\times \T^{k-\rr(H)}
	\end{equation}
	to be the spectrum of simple modules of $\lqh$ obtained from tensor modules of $\lqt{}$ cases. Namely, for $N_i\in\irr L_{q^{m_i}}(2),1\le i\le l,a\in\T^{k-\rr(H)}$, the corresponding tensor module is
	\[N_1\otimes N_2\otimes\cdots\otimes N_l\otimes \C\mathbf{1}_a,\]
	where elements of $Z$ (viewed as Laurent polynomials) act on $\mathbf{1}_a$ via evaluation at $a\in\T^{k-\rr(H)}$. By Jacobson density theorem, the tensor module is a simple $\lqh$-module.

	We take a further look into $\irr\lqt{}$. Bavula \cite{Ba}*{Section 3} classifies all the simple modules of $\lqt{}$ through the approach of quantum Weyl algebra. Rosenberg in \cite{Ro}*{Chapter II, 4.4.5} gives similar results, using the theory of hyperbolic rings. By their results elements of $\irr\lqt{}$ are of the form $\lqt{}/\lqt{}(xy-c),c\in\C^*$ or $\lqt{}/\lqt{}\cap(\C(xy)[x,y]f)$ for some irreducible element $f\in\C(xy)[x,y]$. In practice, we often construct simple $\lqt{}$-modules through twisting the following \emph{basic $\lqt{}$-module} by $\lqt{}$-automorphisms:
	\[V=\C[t,t\inv]\text{ as vector space},xf(t)=f(qt),yf(t)=tf(t),f(t)\in V.\]

	The following properties are easy to check.

	\begin{prop}\label{prop-qtunit}
		(i) Any unit in $\lqh$ is a non-zero scalar multiple of $x^{\aa}$ for some $\aa\in\Z^k$.

		(ii) $\lqh$ as (left) $\lqh$-module is indecomposable.
	\end{prop}

	\section{Representation Theory}
	
	\subsection{Typical $\cqt$-Modules}\label{sec-TypMod}

	In the notation of Remark \ref{rem-quanmat}, $\cqt$ is generated by $x_{11},x_{12},x_{21},x_{22}$ subject to the following relations:
	\begin{equation}\label{eq-relcqt}
		\begin{gathered}
			x_{11}x_{12}=qx_{12}x_{11},x_{11}x_{21}=qx_{21}x_{11},x_{12}x_{22}=qx_{22}x_{12},x_{21}x_{22}=qx_{22}x_{21},\\
			x_{12}x_{21}=x_{21}x_{12},x_{11}x_{22}-x_{22}x_{11}=(q-q\inv)x_{12}x_{21},x_{11}x_{22}-qx_{12}x_{21}=1.
		\end{gathered}
	\end{equation}
	Write the Weyl group $W=\{e,s\}$. Primitive ideals of $\cqt$ are listed by type as follows (cf. \cite{HL1}*{Theorem B.1.1}):
	\begin{equation}\label{eq-sltprim}
		\begin{aligned}
			(e,e)&: P_{(e,e),\gamma}=(x_{12},x_{21},x_{11}-\gamma,x_{22}-\gamma\inv),\gamma\in\C^*,\\
			(s,e)&: P_{(s,e)}=(x_{12}),\\
			(e,s)&: P_{(e,s)}=(x_{21}),\\
			(s,s)&: P_{(s,s),\gamma}=(x_{12}-\gamma x_{21}),\gamma\in\C^*.
		\end{aligned}
	\end{equation}
	From above we see any $(e,e)$ type simple $\cqt$-module is one dimensional, so we neglect simple modules of this type. For each of the types left, we construct several simple $\cqt$-modules, called \emph{typical $\cqt$-modules}.

	\paragraph{$(s,e)$ type} Construct module $M^-(\gamma)$ spanned by basis elements $e_i,i\in\Z$ with module actions
	\begin{equation}\label{eq-semod}
		x_{11}e_{i}=e_{i-1},x_{22}e_i=e_{i+1},x_{12}e_i=0,x_{21}e_i=\gamma q^ie_{i},i\in\Z
	\end{equation}
	for some parameter $\gamma\in\C^*$.

	\paragraph{$(e,s)$ type} Construct module $M^+(\eta)$ spanned by basis elements $e_i,i\in\Z$ with module actions
	\begin{equation}\label{eq-esmod}
		x_{11}e_{i}=e_{i-1},x_{22}e_i=e_{i+1},x_{12}e_i=\eta q^ie_{i},x_{21}e_i=0,i\in\Z
	\end{equation}
	for some parameter $\eta\in\C^*$.

	\paragraph{$(s,s)$ type} Construct \emph{Laurent type module} $M(\gamma,\eta)$ spanned by basis elements $e_i,i\in\Z$ with module actions
	\begin{equation}\label{eq-ssla}
		x_{11}e_{i}=(1+\gamma\eta q^{2i-1})e_{i-1},x_{22}e_i=e_{i+1},x_{12}e_i=\eta q^ie_{i},x_{21}e_i=\gamma q^ie_{i},i\in\Z
	\end{equation}
	for parameters $\gamma,\eta\in\C^*$ satisfying $\gamma\eta\ne -q^{2k+1},\forall k\in\Z$.

	Construct \emph{highest weight type module} $M^h(\eta)$ spanned by basis elements $e_i,i\in\N$ with module actions
	\begin{equation}\label{eq-sshi}
		x_{11}e_{i}=(1-q^{2i})e_{i-1},x_{11}e_0=0,x_{22}e_i=e_{i+1},x_{12}e_i=\eta q^ie_i,x_{21}e_i=-\eta\inv q^{i+1}e_{i},i\in\N
	\end{equation}
	for some parameters $\eta\in\C^*$.

	Construct \emph{lowest weight type module} $M^l(\gamma)$ spanned by basis elements $e_i,i\in\Z_{\le0}$ with module actions
	\begin{equation}\label{eq-sslo}
		x_{11}e_{i}=e_{i-1},x_{22}e_i=(1-q^{2i})e_{i+1},x_{22}e_0=0,x_{12}e_i=\gamma q^{i}e_i,x_{21}e_i=-\gamma\inv q^{i-1}e_{i},i\in\Z_{\le0}
	\end{equation}
	for some parameters $\gamma\in\C^*$.

	All the modules constructed above are simple. This can be shown using a Vandermonde determinant argument, which tells that any nonzero element in the module generates some eigenvector $e_i$ of $x_{12}$ and $x_{21}$, and $e_i$ then generates the whole module. These modules are viewed as ``typical" in the following sense.

	\begin{thm}\label{prop-typmod}
		The typical modules constructed above are the only possible simple $\cqt$-modules with semisimple $\langle x_{12},x_{21}\rangle$-actions.
	\end{thm}

	\begin{proof}
		For the $(s,e)$ type simple $\cqt$-module $V$ with semisimple $\langle x_{12},x_{21}\rangle$-actions, it turns out that $x_{12}V=0$ and $x_{21}$ has an eigenvector in $V$. Denote the eigenvector by $e_0\ne0$, then $x_{21}e_0=\gamma e_0,\gamma\in\C^*$. Let $e_i=x_{22}^ie_0,i\ge0$ and $e_j=x_{11}^{-j}e_0,j<0$. Using (\ref{eq-relcqt}) we see $e_k,k\in\Z$ are nonzero, and linearly independent as eigenvectors of $x_{21}$ with different eigenvalues, thus span a $\cqt$-module with module actions as (\ref{eq-semod}). Therefore $V=\bigoplus_{k\in\Z}\C e_k=M^-(\gamma)$. The $(e,s)$ case is similar.

		For the $(s,s)$ type simple $\cqt$-module $V$ with semisimple $\langle x_{12},x_{21}\rangle$-actions, $x_{12}$ and $x_{21}$'s actions on $V$ differ by a nonzero constant, thus admit a common eigenvector in $V$. Denote the common eigenvector by $\te_0\ne0$, then $x_{12}\te_0=\eta\te_0,x_{21}\te_0=\gamma\te_0,\gamma,\eta\in\C^*$. Using (\ref{eq-relcqt}) we see if $\gamma\eta\ne -q^{2k+1},\forall k\in\Z$, then $x_{22}^k\te_0\ne0,x_{11}^k\te_0\ne0,k\ge0$. Let $e_i=x_{22}^i\te_0,i\ge0$ and $e_j=C_jx_{11}^{-j}\te_0,j<0,C_j\in\C^*$ such that $x_{22}e_k=e_{k+1},k\in\Z$, then $e_k,k\in\Z$ are linearly independent as eigenvectors of $x_{12},x_{21}$, thus span a $\cqt$-module with module actions as (\ref{eq-ssla}). Therefore $V=\bigoplus_{k\in}\C e_k=M(\gamma,\eta)$. If $\gamma\eta=-q^{2k+1}$ for some $k\in\Z$, then $x_{22}^l\te_0=0$ or $x_{11}^{l'}\te_0=0,\exists l,l'\ge0$. In this case $\te_0$ generates a module as (\ref{eq-sshi}) or (\ref{eq-sslo}), using a similar treatment.
	\end{proof}

	\subsection{Tensor Modules and Weight Strings}

	For $1\le i\le n$, denote by $\cqb{i}=\C_{q_i}[SL_2]_i,\cqbp{i}=\cqb{i}/(x_{21})$ and $\cqbm{i}=\cqb{i}/(x_{12})$ the quantum coordinate algebras corresponding to $U_i,U(i)$ and $U(-i)$ respectively. Then $M^\pm(*)$ is in fact $\C_{q_i}[B_i^\pm]$-modules. Furthermore, one can embed $\C_{q_i}[B^\pm]$ into $\lqt{i}$ via $x_{11}\mapsto x,x_{22}\mapsto x\inv$ and $x_{12}\mapsto y$ (or $x_{21}\mapsto y$). Denote the embeddings by $\iota_i^\pm$. In this way $M^\pm(*)$ can be made into $\lqt{i}$-modules.
	
	The natural embedding $U_i\hookrightarrow U$ induces the surjective restriction map $\pi_i:\cqg\twoheadrightarrow\cqb{i}$. Denote the composition of $\pi_i$ with quotient map $\cqb{i}\twoheadrightarrow\C_{q_i}[B_i^\pm]$ by $\pi_i^\pm$. Via $\pi_i$ or $\pi_i^\pm$ one can lift $M^\bullet(*)$ to simple $\cqg$-modules, and denote them by $M_i^{\bullet}(*)$.

	First by Levendorski\u{\i} and Soibelman \cite{LS}, then by Joseph \cite{Jo}, Saito \cite{Sa} and Tanisaki \cite{Ta}, the following theorem concerns tensoring lifted $(s,s)$-typical modules of highest weight type:
	\begin{thm}\label{thm-HWM}
		Let $\ii=\iim\in\I_w$, fix $\gamma\in\C^*$, then $M_{{i_1}}^h(\gamma)\otimes M_{{i_2}}^h(\gamma)\otimes\cdots\otimes M_{{i_m}}^h(\gamma)$ is a simple $\C_q[G]$-module depending only on $w$ (up to isomorphism), and it is of type $(w,w)$.
	\end{thm}
	\begin{rem}
		The character formula of $M_{{i_1}}^h(\gamma)\otimes M_{{i_2}}^h(\gamma)\otimes\cdots\otimes M_{{i_m}}^h(\gamma)$ is explicitly calculated, see \cite{LS,Jo}.
	\end{rem}

	Let $w\in W$, we want to generalize Theorem \ref{thm-HWM} by considering the tensor module
	\begin{equation}\label{eq-Mii}
		M_{\ii}:=M_{i_1}\otimes M_{i_2}\otimes\cdots\otimes M_{i_m},\ii=\iim\in\I_w,
	\end{equation}
	where $M_{i_k}$ stands for one of $M_{i_k}^h(\gamma_k)$, $M_{i_k}^l(\eta_k)$ and $M_{i_k}(\gamma_k,\eta_k)$, $1\le k\le m$ and $\gamma_k,\eta_k$ satisfies conditions in (\ref{eq-ssla})-(\ref{eq-sslo}). $M_{\ii}$ is a natural $\cqb{i_1}\otimes\cqb{i_2}\otimes\cdots\otimes\cqb{i_m}$-module, and $\cqg$ acts on $M_{\ii}$ via
	\[\pi_{\ii}=(\pi_{i_1}\otimes\pi_{i_2}\otimes\cdots\otimes\pi_{i_m})\circ\Delta^{(m-1)}.\]

	Furthermore, let $(\ww)\in\wtw$, we consider the tensor module
	\begin{equation}\label{eq-Mtii}
		\tmii:=\tm_{{\ti_1}}\otimes\tm_{{\ti_2}}\otimes\cdots\otimes\tm_{{\ti_m}},\tii=\tiim\in\I_{\ww},
	\end{equation}
	where $\tm_{\ti_k}$ stands for $M_{|\ti_k|}^{\sgn(\ti_k)}(\gamma_k)$, $1\le k\le m$, $\sgn(\ti_k)$ is the sign of $\ti_k$ and $\gamma_k\in\C^*$. $\tmii$ is a natural $\lqi:=\lqt{|\ti_1|}\otimes\lqt{|\ti_2|}\otimes\cdots\otimes\lqt{|\ti_m|}$-module, and $\cqg$ acts on $\tmii$ via
	\[\tpi_{\tii}=(\iota_{|\ti_1|}^{\sgn(\ti_1)}\otimes\iota_{|\ti_2|}^{\sgn(\ti_2)}\otimes\cdots\otimes\iota_{|\ti_m|}^{\sgn(\ti_m)})\circ(\pi_{|\ti_1|}^{\sgn(\ti_1)}\otimes\pi_{|\ti_2|}^{\sgn(\ti_2)}\otimes\cdots\otimes\pi_{|\ti_m|}^{\sgn(\ti_m)})\circ\Delta^{(m-1)}.\]
	Denote $\{x^{\aa}y^{\bb}|\aa,\bb\in\Z^m\}$ a basis of $\lqi$, where 
	\[x^{\aa}=x_1^{a_1}\otimes x_2^{a_2}\otimes\cdots\otimes x_m^{a_m},\aa=(a_1,a_2,\cdots,a_m)\tr\in\Z^m,\]
	\[y^{\bb}=y_1^{b_1}\otimes y_2^{b_2}\otimes\cdots\otimes y_m^{b_m},\bb=(b_1,b_2,\cdots,b_m)\tr\in\Z^m,\]
	and $x_k,y_k$ are generators of $\lqt{|\ti_k|},1\le k\le m$, c.f. Section \ref{sec-qt}.

	To analyze structure of tensor modules, first we deal with the $G=SL_2$ case. Let $V(l)$ be the $U_q(\slt)$-module with highest weight $l\in\N$. Let $v_{li}$ be some weight vector of weight $i\in\{l,l-2,\cdots,-l\}$ in $V(l)$, such that $v_{l,-l}=T(v_{ll})=(-q)^lF^{(l)}v_{ll}$. Let $\{v_{li}^*\}$ be the dual basis of $\{v_{li}\}$. Denote $x_{\pm}^k=x_{11}^k$, if $k\ge0$, and $x_{\pm}^k=x_{22}^{-k}$ if $k<0$. Let $x_{ij}^{(l)}=c_{v_{li}^*,v_{lj}}^{V(l)}\in\cqt$, we calculate $x_{ij}^{(l)}$ explicitly in the following lemma.

	\begin{lem}\label{lem-xlij}
		Let $l\in\N$, $i,j\in\{l,l-2,\cdots,-l\}$.

		(i) There exist nonzero polynomials $P_{lij}\in\C[u_1,u_2]$ such that
		\[x_{ij}^{(l)}=P_{lij}(x_{12},x_{21})x_{\pm}^{(i+j)/2}.\]

		(ii) There exist scalars $c_{lij}^\pm\in\C$ such that
		\[\iota^\pm\circ\pi^\pm(x_{ij}^{(l)})=c_{lij}^\pm x^{(i+j)/2}y^{|i-j|/2},\]
		where $c_{lij}^\pm$ is nonzero if and only if $\pm(i-j)\ge0$.

		(iii) $x_{ll}^{(l)}=x_{11}^l,x_{l,-l}^{(l)}=x_{12}^l,x_{-l,l}^{(l)}=x_{21}^l,x_{-l,-l}^{(l)}=x_{22}^l$.
	\end{lem}

	\begin{proof}
		Embed $V(l)$ into $V(1)^{\otimes l}$. Then $v_{lj}$ and $v_{li}^*$ are of the form
		\[v_{lj}=\sum_{\tiny j_k=\pm 1,\Sigma j_k=j}c_{j_1j_2\cdots j_l}v_{1j_1}\otimes v_{1j_2}\otimes\cdots\otimes v_{1j_l},\]
		\[v_{li}^*=\sum_{\tiny i_k=\pm 1,\Sigma i_k=i}c'_{i_1i_2\cdots i_l}v_{1i_1}^*\otimes v_{1i_2}^*\otimes\cdots\otimes v_{1i_l}^*,\]
		for some $c_{j_1j_2\cdots j_l},c'_{i_1i_2\cdots i_l}\in\C$ depending on $v_{lj}$ and $v_{li}^*$. Thus
		\[x_{ij}^{(l)}=\sum c_{j_1j_2\cdots j_l}c'_{i_1i_2\cdots i_l}c_{v_{1i_1}^*,v_{1j_1}}^{V(1)}c_{v_{1i_2}^*,v_{1j_2}}^{V(1)}\cdots c_{v_{1i_l}^*,v_{1j_l}}^{V(1)},\]
		where the sum is over $i_k=\pm 1,\sum i_k=i,j_k=\pm 1,\sum j_k=j$, and $c_{v_{1i_k}^*,v_{1j_k}}^{V(1)}$ is a non-zero scalar multiple of $x_{(3-i_k)/2,(3-j_k)/2}$. For some arbitrary summand term, denote by $l_{ab}$ the number of appearances of $x_{ab}$ in it, $a,b=1,2$. Then we see
		\begin{gather*}
			l_{11}+l_{12}=|\{k|i_k=1\}|=(l+i)/2,\\
			l_{11}+l_{21}=|\{k|j_k=1\}|=(l+j)/2,\\
			l_{11}+l_{12}+l_{21}+l_{22}=l,
		\end{gather*}
		so we get $l_{11}-l_{22}=(i+j)/2$. Use commutative relations (\ref{eq-relcqt}), in each summand term we move $x_{11}$'s and $x_{22}$'s to the right side of $x_{12}$'s and $x_{21}$'s, then use $x_{11}x_{22}=1+qx_{12}x_{21}$ or $x_{22}x_{11}=1+q\inv x_{12}x_{21}$ to eliminate $x_{11}$'s or $x_{22}$'s. This implies (i).
		
		If we let $l_{21}=0$ (resp. $l_{12}=0$), then we get $l_{12}=(i-j)/2$ (resp. $l_{21}=(j-i)/2$), which is possible if and only if $i-j\ge0$ (resp. $j-i\ge0$). This implies (ii).

		For (iii), notice that if we choose $v_{ll}=v_{11}\otimes v_{11}\otimes\cdots\otimes v_{11}$, then by Corollary \ref{cor-Twaction}
		\[v_{l,-l}=T(v_{ll})=T(v_{11})\otimes T(v_{11})\otimes\cdots\otimes T(v_{11})=v_{1,-1}\otimes v_{1,-1}\otimes\cdots\otimes v_{1,-1},\]
		so (iii) follows.
	\end{proof}

	Back to the $G$-semisimple case. We introduce the concept of \emph{weight strings}.

	\begin{defn}\label{defn-wtstr}
		Let $w\in W,\ii=\iim\in\I_w,\nu,\mu\in P$. A tuple of weights $\muu=(\mu_0,\mu_1,\cdots,\mu_m)\in P^{m+1}$ is called a \emph{weight string from $\nu$ to $\mu$ of type $\ii$}, if (i) $\mu_0=\nu$, $\mu_m=\mu$, (ii) $\mu_{k-1}-\mu_k\in\Z\alpha_{i_k},1\le k\le m$. Denote the set of these weight strings by $\pp^{\pm}_{\ii,\nu,\mu}$.

		Similarly, for $(\ww)\in\wtw,\tii=\tiim\in\I_{\ww}$ and $\nu,\mu\in P$, a tuple of weights $\muu=(\mu_0,\mu_1,\cdots,\mu_m)\in P^m$ is called a \emph{weight string from $\nu$ to $\mu$ of type $\tii$}, if (i) $\mu_0=\nu$, $\mu_m=\mu$, (ii) $\mu_{k-1}-\mu_k\in\N\,\sgn(\ti_k)\alpha_{\ti_k},1\le k\le m$. Denote the set of these weight strings by $\pp_{\tii,\nu,\mu}$.
	\end{defn}

	\begin{rem}
		View $P^m$ as $\Z$-lattice with component-wise addition and scalar multiplication. Then $\pp^{\pm}_{\ii,\nu,\mu}$ is a sublattice of $P^m$, and $\pp_{\tii,\nu,\mu}$ is a submonoid of $P^m$.
	\end{rem}

	We calculate images of matrix coefficients under $\pi_\ii$ or $\tpi_{\tii}$ with weight strings.

	\begin{prop}\label{prop-pii}
		Let $V\in\cc,f\in (V^*)_{-\nu},v\in V_\mu,\nu,\mu\in P$.

		(i) For $w\in W,\ii=\iim\in\I_w$, there exist polynomials $P_{\muu,k}\in\C[u_1,u_2],\muu\in\pp^{\pm}_{\ii,\nu,\mu},1\le k\le m$ depending on $f,v$ such that
		\begin{equation}\label{eq-pii}
			\pi_{\ii}(c_{f,v}^V)=\sum_{\muu\in\pp^{\pm}_{\ii,\nu,\mu}}\bigotimes_{k=1}^m P_{\muu,k}(x_{12},x_{21})x_\pm^{(\mu_{k-1}+\mu_k,\alpha_{i_k}^\vee)/2}.
		\end{equation}

		(ii) For $(\ww)\in\wtw,\tii=\tiim\in\I_{\ww}$, there exist scalars $c_{\muu}\in\C,\muu\in\pp_{\tii,\nu,\mu}$ depending on $f,v$ such that
		\begin{equation}\label{eq-tpii}
			\tpi_{\tii}(c_{f,v}^V)=\sum_{\muu\in\pp_{\tii,\nu,\mu}}c_{\muu}x^{\aa(\muu)}y^{\bb(\muu)},
		\end{equation}
		where 
		\[\aa(\muu)=(a_{\muu,1},a_{\muu,2},\cdots,a_{\muu,m})\tr,a_{\muu,k}=(\mu_{k-1}+\mu_k,\alpha_{|\ti_k|}^\vee)/2,1\le k\le m,\]
		\[\bb(\muu)=(b_{\muu,1},b_{\muu,2},\cdots,b_{\muu,m})\tr,b_{\muu,k}=|(\mu_{k-1}-\mu_k,\alpha_{|\ti_k|}^\vee)|/2,1\le k\le m.\]
	\end{prop}

	\begin{proof}
		(i) Use induction on $m=l(\ii)$. When $m=1,\ii=(i_1)$, $\pp^{\pm}_{\ii,\nu,\mu}$ is non-empty (i.e. $\{(\nu,\mu)\}$) if and only if $\nu-\mu=j_1\alpha_{i_1}$ for some $j_1\in\Z$. So (\ref{eq-pii}) follows from Lemma \ref{lem-xlij} (i). Suppose (\ref{eq-pii}) is true for $\ii'=(i_1,i_2,\cdots,i_{m-1}),m\ge2$, we prove it's true for $\ii$. Let $v_1,v_2,\cdots,v_t$ be a basis of $U_{i_m}v$ formed by weight vectors. Extend $\{v_i\}_{i=1}^t$ to $\{v_i\}_{i=1}^l$ which form a basis of $V$, and let $\{v_i^*\}_{i=1}^l$ be the dual basis. Note $\pi_{i_m}(c_{v_i^*,v}^V)=0,i>t$. By definition of $\pi_{\ii}$ and (\ref{eq-cqghopf}) we calculate
		\begin{align*}
			\pi_{\ii}(c_{f,v}^V)&=(\pi_{i_1}\otimes\pi_{i_2}\otimes\cdots\otimes\pi_{i_m})\circ(\Delta^{(m-2)}\otimes 1)(\Delta(c_{f,v}))\\
			&=(\pi_{i_1}\otimes\pi_{i_2}\otimes\cdots\otimes\pi_{i_m})\circ(\Delta^{(m-2)}\otimes 1)(\sum_{i=1}^l c_{f,v_i}^V\otimes c_{v_i^*,v}^V)\\
			&=\sum_{i=1}^t\pi_{\ii'}(c_{f,v_i}^V)\otimes\pi_{i_m}(c_{v_i^*,v}^V).
		\end{align*}
		In the result above, substitute $\pi_{\ii'}(c_{f,v_i}^V)$ with induction hypothesis and $\pi_{i_m}(c_{v_i^*,v}^V)$ with Lemma \ref{lem-xlij} (i), we get the right side of (\ref{eq-pii}) for $\ii=\iim$.

		(ii) The proof is similar to that of (i), using Lemma \ref{lem-xlij} (ii).
	\end{proof}

	\begin{rem}\label{rem-pii}
		Denote $X_{i}^j=E_i^j$ if $j\ge0$, and $X_{i}^j=F_i^{-j}$ if $j<0$. In (\ref{eq-pii}) or (\ref{eq-tpii}), for weight string $\muu=(\mu_0,\mu_1,\cdots,\mu_m)$ such that $\mu_{k-1}-\mu_k=j_k\alpha_{|\ti_k|}$, the summand term corresponding to $\muu$ is nonzero if and only if $X_{|\ti_k|}^{j_k}X_{|\ti_{k+1}|}^{j_{k+1}}\cdots X_{|\ti_m|}^{j_m}v\ne 0,1\le k\le m$ and $f(X_{|\ti_1|}^{j_1}X_{|\ti_2|}^{j_2}\cdots X_{|\ti_m|}^{j_m}v)\ne0$.
	\end{rem}

	\subsection{Modules of $(w,w)$ Type}\label{sec-ss}

	Fix some $w\in W$ and $\ii=\iim\in\I_w$, we prove the following theorem.

	\begin{thm}\label{thm-wwmod}
		The tensor module $M_{\ii}$ defined in (\ref{eq-Mii}) is simple of $(w,w)$ type.
	\end{thm}
	
	Denote $w_{>k}=s_{i_{k+1}}s_{i_{k+2}}\cdots s_{i_m}$ for $0\le k\le m-1$ and $w_{>m}=e$. Let $a_k=(w_{>k}(\lambda),\alpha_{i_k}^\vee)\ge0,1\le k\le m$. We are interested in images of normal elements and its variants under $\pi_\ii$. Define \emph{$\ii$-normal elements} to be $\np{w}{\lambda},\nm{w}{\lambda},\lambda\in P^+$, and \emph{quasi-$\ii$-normal elements} to be
	\[\np{w}{\lambda}(j):=(F_{i_1}^{(j)}\otimes 1)\np{w}{\lambda},\nm{w}{\lambda}(j)=(E_{i_1}^{(j)}\otimes 1)\nm{w}{\lambda},\]
	where $\lambda\in P^+,0\le j\le(w_{>1}(\lambda),\alpha_{i_1}^\vee)$.

	\begin{prop}\label{prop-PiNm}
		Let $\lambda\in P^+$.

		(i) There exist $c_j,c_j'\in\C^*,0\le j\le a_1$ such that
		\[\pi_{\ii}(\np{w}{\lambda}(j))=c_jx_{21}^{a_1-j}x_{11}^j\otimes x_{21}^{a_2}\otimes\cdots\otimes x_{21}^{a_m},\pi_{\ii}(\nm{w}{\lambda}(j))=c_j'x_{12}^{a_1-j}x_{22}^j\otimes x_{12}^{a_2}\otimes\cdots\otimes x_{12}^{a_m}.\]

		(ii) In particular,
		\[\pi_{\ii}(\np{w}{\lambda})=x_{21}^{a_1}\otimes x_{21}^{a_2}\otimes\cdots\otimes x_{21}^{a_m},\pi_{\ii}(\nm{w}{\lambda})=x_{12}^{a_1}\otimes x_{12}^{a_2}\otimes\cdots\otimes x_{12}^{a_m},\]
		\[\pi_{\ii}(\np{w_{>1}}{\lambda})=x_{11}^{a_1}\otimes x_{21}^{a_2}\otimes\cdots\otimes x_{21}^{a_m},\pi_{\ii}(\nm{w_{>1}}{\lambda})=x_{22}^{a_1}\otimes x_{12}^{a_2}\otimes\cdots\otimes x_{12}^{a_m}.\]
	\end{prop}

	\begin{proof}
		(i) By Corollary \ref{cor-DemMod} (i), $U_{i_{k+1}}U_{i_{k+2}}\cdots U_{i_m}v_\lambda=U^+v_{w_{>k}(\lambda)},0\le k\le m$. Apply Proposition \ref{prop-pii} (i) to calculate $\pi_{\ii}(\np{w}{\lambda}(j))$, the only nonzero summand term in the right side of (\ref{eq-pii}) is the term corresponding to the weight string $(w(\lambda)+j\alpha_{i_1},w_{>1}(\lambda),\cdots,w_{>m-1}(\lambda),\lambda)$, c.f. Remark \ref{rem-pii}. This implies the desired formula for $\pi_{\ii}(\np{w}{\lambda}(j))$. The calcuation for $\pi_{\ii}(\nm{w}{\lambda}(j))$ is similar.
		
		(ii) Similar as (i), by applying Proposition \ref{prop-pii} (i),
		\[\pi_{\ii}(\np{w}{\lambda})=\bigotimes_{k=1}^m\pi_{i_k}(c_{v_{w_{>k-1}(\lambda)}^*,v_{w_{>k}(\lambda)}}^{V(\lambda)}),\]
		\[\pi_{\ii}(\np{w_2}{\lambda})=\pi_{i_1}(c_{v_{w_2(\lambda)}^*,v_{w_2(\lambda)}}^{V(\lambda)})\otimes\bigotimes_{k=2}^m\pi_{i_k}(c_{v_{w_{>k-1}(\lambda)}^*,v_{w_{>k}(\lambda)}}^{V(\lambda)}).\]
		By Lemma \ref{lem-xlij} (iii), $\pi_{i_k}(c_{v_{w_{>k-1}(\lambda)}^*,v_{w_{>k}(\lambda)}}^{V(\lambda)})=x_{21}^{a_k}$, $\pi_{i_1}(c_{v_{w_{>1}(\lambda)}^*,v_{w_{>1}(\lambda)}}^{V(\lambda)})=x_{11}^{a_1}$. The calculation for $\pi_{\ii}(\nm{w}{\lambda})$ and $\pi_{\ii}(\nm{w_{>1}}{\lambda})$ is similar.
	\end{proof}

	\begin{cor}\label{cor-wwtype}
		$\ker\pi_{\ii}=J_{w,w}$.
	\end{cor}

	\begin{proof}
		By Corollary \ref{cor-DemMod}, elements of $J_w^\pm$ take value 0 on $U_{\ii}$, so $J_{w,w}\subset\ker\pi_{\ii}$. Note that the codomain of $\pi_{\ii}$ is an integral domain, and $\ker\pi_{\ii}$ is $H$-invariant, thus $\ker\pi_{\ii}$ is an $H$-prime containing $J_{w,w}$. Proposition \ref{prop-PiNm} shows $\ker\pi_{\ii}\cap E_w^\pm=\varnothing$. Combining Theorem \ref{thm-Hprime} and Theorem \ref{thm-stratification}, $\ker\pi_{\ii}$ must be $J_{w,w}$.
	\end{proof}

	Before stating the proof of Theorem \ref{thm-wwmod}, we introduce some notations. For $M_{i_k}$ being $M_{i_k}^h(\gamma_k)$, $M_{i_k}^l(\eta_k)$ or $M_{i_k}(\gamma_k,\eta_k)$, $1\le k\le m$, let $\zz_k$ be the subscript set $\N,\Z_{\le0}$ or $\Z$, and let $b_k=-q_{i_k},-q_{i_k}\inv$ or $\gamma_k\eta_k$  respectively, c.f. (\ref{eq-ssla})-(\ref{eq-sslo}). For $\nn=\nnm\tr\in\zz_{\ii}:=\zz_1\times\zz_2\times\cdots\times\zz_m$, denote $e_{\nn}=e_{n_1}\otimes e_{n_2}\otimes\cdots\otimes e_{n_m}\in M_{\ii}$. Such $e_{\nn}$'s form a basis of $M_{\ii}$.

	\begin{proof}[Proof of Theorem \ref{thm-wwmod}]
		We use induction on $l(\ii)$ and show that if $M_{\ii'}$ is simple, where $\ii'=(i_2,i_3,\cdots,i_m)$, then $M_{\ii}$ is simple. 
		Let $\ee=\sum_{\nn\in\zz_{\ii}}c_{\nn}e_{\nn}$ be any nonzero element of $M_{\ii}$, where $c_{\nn}\in\C$ is all but finitely 0. We want to show $\cqg\ee=M_{\ii}$ with the following steps.
		
		\emph{Step 1.} Rewrite $\ee=\sum_{k\in\zz_1}e_k\otimes\ee'_k,\ee'_k\in M_{\ii'}$. Claim that $e_k\otimes\ee'_k\in\cqg\ee,\forall k\in\zz_1$. Choose some $\lambda\in P^+$ such that $a_1=(w_{>1}(\lambda),\alpha_{i_1}^\vee)>0$. Apply Propostion \ref{prop-PiNm},
		\[\np{w}{\lambda}\nm{w}{\lambda}e_{\nn}=C_{\nn}e_{\nn},\]
		\[\np{w}{\lambda}(1)\nm{w}{\lambda}(1)e_{\nn}=C_{\nn}'e_{\nn},\]
		where
		\[C_{\nn}=b_1^{a_1}b_2^{a_2}\cdots b_m^{a_m}q_{i_1}^{2a_1n_1}q_{i_2}^{2a_2n_2}\cdots q_{i_m}^{2a_mn_m},\]
		\[C_{\nn}'=c_1c_1'(1+b_1q_{i_1}^{2n_1+1})b_1^{a_1-1}b_2^{a_2}\cdots b_m^{a_m}q_{i_1}^{2(a_1-1)(n_1+1)}q_{i_2}^{2a_2n_2}\cdots q_{i_m}^{2a_mn_m}.\]
		Notice $C_{\nn}'/C_{\nn}=c_1c_1'q_{i_1}^{a_1}(1+b_1\inv q_{i_1}^{-2n_1-1})$, which is an injective function on $n_1$ since $q$ is generic. So for different $n_1$, $C_{\nn}$ or $C_{\nn}'$ must be different. The claim is then obtained using a Vandermonde determinant argument.

		\emph{Step 2.} For some nonzero $e_k\otimes\ee_k'$ in \emph{Step 1}, we want to show
		\[(1\otimes\pi_{\ii'}(\cqg))(e_k\otimes\ee_k')\subset\cqg(e_k\otimes\ee_k'),\]
		which leads to $e_k\otimes M_{\ii'}\subset\cqg (e_k\otimes\ee_k')$ by induction hypothesis. In fact, for any matrix coefficient $c_{f,v}^V$, where $V\in\cc,f\in(V^*)_{-\nu},v\in V_\mu,\nu,\mu\in P$, consider $c_{T_{i_1}(f),v}^V$'s action on $e_k\otimes\ee_k'$. Notice that the set $\pp_{\ii,s_{i_1}(\nu),\mu}$ of weight strings from $s_{i_1}(\nu)$ to $\mu$ can be decomposed into $\pp_1\sqcup\pp_2$, where
		\[\pp_1=\{(s_{i_1}(\nu),\nu,\mu_2,\cdots,\mu)|(\nu,\mu_2,\cdots,\mu)\in\pp_{\ii',\nu,\mu}\},\]
		\[\pp_2=\{(s_{i_1}(\nu),\mu_1,\cdots,\mu)\in\pp_{\ii,s_{i_1}(\nu),\mu}|\mu_1\ne\nu\}.\]
		Apply Proposition \ref{prop-pii} (i), one gets
		\begin{align*}
			\pi_{\ii}&(c_{T_{i_1}(f),v}^V)=P_0(x_{12},x_{21})\otimes\pi_{\ii'}(c_{f,v}^V)\\
			&+\sum_{\muu\in\pp_2}P_{\muu,1}(x_{12},x_{21})x_\pm^{a_{\muu,1}}\otimes P_{\muu,2}(x_{12},x_{21})x_\pm^{a_{\muu,2}}\otimes\cdots\otimes P_{\muu,m}(x_{12},x_{21})x_\pm^{a_{\muu,m}},
		\end{align*}
		where $P_0,P_{\muu,1},P_{\muu,2},\cdots,P_{\muu,m}\in\C[u_1,u_2],P\ne0,a_{\muu,1},a_{\muu,2},\cdots,a_{\muu,m}\in\Z,a_{\muu,1}\ne0$. Thus
		\[c_{T_{i_1}(f),v}^V(e_k\otimes\ee_k')=ce_k\otimes\pi_{\ii'}(c_{f,v}^V)\ee_k'+\sum_{k'\ne k}e_{k'}\otimes \ee_{k'}'',\]
		where $c\in\C^*,\ee_{k'}''\in M_{\ii'}$. Due to \emph{Step 1}, $e_k\otimes\pi_{\ii'}(c_{f,v}^V)\ee_k'\in\cqg (e_k\otimes\ee_k')$.

		\emph{Step 3.} For some nonzero $e_k\otimes\ee_k'$ in \emph{Step 1}, applying $c_{w,\lambda}^\pm(1)$'action on $e_k\otimes\ee_k'$, one gets $e_l\otimes\ee_l'''\in\cqg(e_k\otimes\ee_k'),\forall l\in\zz_1$, for some nonzero $\ee_l'''\in M_{\ii'}$, c.f. Proposition \ref{prop-PiNm}.

		\emph{Step 4.} Combining \emph{Step 1-3}, one gets $M_{\ii}=\sum_{l\in\zz_1}e_{l}\otimes M_{\ii'}\subset\cqg\ee$, i.e. $M_{\ii}$ is simple. Corollary \ref{cor-wwtype} tells $M_{\ii}$ is $(w,w)$ type. So we obtain the theorem.
	\end{proof}
	
	\subsection{Localization of $\cqg_{\ww}$}

	Fix some $(\ww)\in\wtw$ and $\tii=\tiim\in\I_{\ww}$. In the following subsections we study the tensor module $\tmii$ defined by (\ref{eq-Mtii}). Let $w_{\le k}^\pm,w_{>k}^\pm,1\le k\le m$ be
	\[w_{\le k}^\pm=\prod_{\pm i_j>0,j\le k}s_{|i_j|},w_{>k}^\pm=\prod_{\pm i_j>0,j>k}s_{|i_j|}.\]
	Define \emph{left quasi-$\tii$-normal elements} to be
	\begin{equation}\label{eq-quasinormL}
		\nl{\tii}{\lambda}(j)=\begin{cases}
			(F_{|i_1|}^{(j)}\otimes 1)\np{w_1}{\lambda},\lambda\in P^+,0\le j\le(\wm{1}(\lambda),\alpha_{|i_1|}^\vee), & \text{if $i_1<0$},\\
			(E_{|i_1|}^{(j)}\otimes 1)\nm{w_2}{\lambda},\lambda\in P^+,0\le j\le(\wp{1}(\lambda),\alpha_{|i_1|}^\vee), & \text{if $i_1>0$},
		\end{cases}
	\end{equation}
	and \emph{right quasi-$\tii$-normal elements} to be
	\begin{equation}\label{eq-quasinormR}
		\nr{\tii}{\lambda}(j)=\begin{cases}
			(1\otimes F_{|i_1|}^{(j)})\np{w_1}{\lambda},\lambda\in P^+,0\le j\le(\lambda,\alpha_{|i_m|}^\vee), & \text{if $i_m<0$},\\
			(1\otimes E_{|i_1|}^{(j)})\nm{w_2}{\lambda},\lambda\in P^+,0\le j\le(\lambda,\alpha_{|i_m|}^\vee), & \text{if $i_m>0$}.
		\end{cases}
	\end{equation}
	We deduce their actions on $\tmii$. Recall in Proposition \ref{prop-pii} we define
	\begin{equation}\label{eq-amubmu}
		\begin{gathered}
			\aa(\muu)=(a_{\muu,1},a_{\muu,2},\cdots,a_{\muu,m})\tr,a_{\muu,k}=(\mu_{k-1}+\mu_k,\alpha_{|\ti_k|}^\vee)/2,1\le k\le m,\\
			\bb(\muu)=(b_{\muu,1},b_{\muu,2},\cdots,b_{\muu,m})\tr,b_{\muu,k}=|(\mu_{k-1}-\mu_k,\alpha_{|\ti_k|}^\vee)|/2,1\le k\le m,
		\end{gathered}
	\end{equation}
	for weight string $\muu=(\mu_0,\mu_1,\cdots,\mu_m)$ of type $\tii$. Denote $I(\muu)=x^{\aa(\muu)}y^{\bb(\muu)}\in \lqi$ and
	\[D_{\tii}=\frac{1}{2}\diag((\alpha_{|i_1|},\alpha_{|i_1|}),(\alpha_{|i_2|},\alpha_{|i_2|}),\cdots,(\alpha_{|i_m|},\alpha_{|i_m|})).\]

	\begin{prop}\label{prop-tpinm}
		Let $\lambda\in P^+$. Keep notations above. There exist $c_j,c_j'\in\C^*$ such that
		\[\tpi_{\tii}(\nl{\tii}{\lambda}(j))=c_jI(\muu_{\tii,\lambda}^l(j)),\tpi_{\tii}(\nr{\tii}{\lambda}(j))=c_j'I(\muu_{\tii,\lambda}^r(j))\]
		where $\muu_{\tii,\lambda}^l(j),\muu_{\tii,\lambda}^r(j)$ are weight strings of type $\tii$ defined as
		\[\muu_{\tii,\lambda}^l(j)=\begin{cases}
			(w_1(\lambda)+j\alpha_{|i_1|},\wm{1}(\lambda),\cdots,\wm{m-1}(\lambda),\lambda), & \text{if $i_1<0$},\\
			(-w_2(\lambda)-j\alpha_{|i_1|},-\wp{1}(\lambda),\cdots,-\wp{m-1}(\lambda),-\lambda), & \text{if $i_1>0$},
		\end{cases}\]
		\[\muu_{\tii,\lambda}^r(j)=\begin{cases}
			(w_1(\lambda),\wm{1}(\lambda),\cdots,\wm{m-1}(\lambda),\lambda-j\alpha_{|i_m|}), & \text{if $i_m<0$},\\
			(-w_2(\lambda),-\wp{1}(\lambda),\cdots,-\wp{m-1}(\lambda),-\lambda+j\alpha_{|i_m|}). & \text{if $i_m>0$}.
		\end{cases}\]
	\end{prop}

	\begin{proof}
		We check for $\nl{\tii}{\lambda}(j)$ when $i_1<0$. By Corollary \ref{cor-DemMod} (ii), $U(i_{k+1})U(i_{k+2})\cdots U(i_m)v_\lambda=U^+v_{\wm{k}(\lambda)},0\le k\le m$. Apply Proposition \ref{prop-pii} (ii) to calculate $\tpi_{\tii}(\nl{\tii}{\lambda}(j))$, the only nonzero summand term in the right side of (\ref{eq-tpii}) is the term corresponding to the weight string $(w_1(\lambda)+j\alpha_{|i_1|},\wm{1}(\lambda),\cdots,\wm{m-1}(\lambda),\lambda)$, c.f. Remark \ref{rem-pii}. This implies the desired formula for $\tpi_{\tii}(\nl{\tii}{\lambda}(j))$ when $i_1<0$. The computation of other cases is similar.
	\end{proof}

	\begin{cor}\label{cor-w1w2type}
		$\ker\tpi_{\tii}=J_{\ww}$.
	\end{cor}

	\begin{proof}
		By Corollary \ref{cor-DemMod} (ii), elements of $J_{w_1}^+$ and $J_{w_2}^-$ take value 0 on $U(\tii)$, so $J_{\ww}\subset\ker\tpi_{\tii}$. Note that the codomain of $\tpi_{\tii}$ is an integral domain, and $\ker\tpi_{\tii}$ is $H$-invariant, thus $\ker\tpi_{\tii}$ is an $H$-prime containing $J_{\ww}$. Proposition \ref{prop-tpinm} shows $\ker\tpi_{\tii}\cap E_{\ww}=\varnothing$. Combining Theorem \ref{thm-Hprime} and Theorem \ref{thm-stratification}, $\ker\tpi_{\tii}$ must be $J_{\ww}$.
	\end{proof}

	Denote by $R_{\tii}\subset \lqi$ the image of $\cqg$ under $\tpi_{\tii}$. The corollary above shows $R_{\tii}\simeq\cqg/J_{\ww}$. Denote by $\lqi^\times=\{cx^{\aa}y^{\bb}|\aa,\bb\in\Z^m,c\in\C^*\}$ the set of units in $\lqi$ (c.f. Proposition \ref{prop-qtunit}). For any subalgebra $A\subset \lqi$, denote $A^\times=A\cap \lqi^\times$. In particular, denote $S_{\tii}=R_{\tii}^\times$. Proposition $\ref{prop-tpinm}$ tells (left and right) quasi-$\tii$-normal elements are in $S_{\tii}$. Before considering the localization of $R_{\tii}$ with respect to $S_{\tii}$, we need the following lemma.

	\begin{lem}\label{lem-ore}
		Let $R$ be an associative algebra over field $\F$, and $q\in\F^*$ is not a root of unity. Given $a,b\in R$ such that $b=\sum b_i$ (finite sum), $ab_i=q^ib_ia$. Then $a$ is an Ore element of $\F\langle a,b\rangle$ and $\F\langle a,b\rangle[a\inv]$=$\F\langle a^\pm,b_i\rangle$.
	\end{lem}

	\begin{proof}
		Let $N\in\N$ be greater than $|\{i|b_i\ne0\}|$. Write
		\[a^kba^{N-k}=\sum q^{ki}b_ia^N,a^{N-k}ba^k=\sum q^{-ki}a^Nb_i,0\le k\le N.\]
		So a Vandermonde determinant argument shows $b_ia^N,a^Nb_i\in\F\langle a,b\rangle$, and there exist $b',b''\in\F\langle a,b\rangle$ such that $a^{N+1}b=b'a,ba^{N+1}=ab''$. Thus $a$ satisfies the left and right Ore conditions. The rest of the lemma follows consequently.
	\end{proof}

	\begin{rem}
		In settings of the lemma above, we call two elements $a,b\in R$ are \emph{$q$-commute} if $ab=q^kba,k\in\Z$. We call $k$ the \emph{$q$-commute index of $a,b$} and denote it by $[a,b]_q=k$.
	\end{rem}

	\begin{cor}\label{cor-ore}
		$S_{\tii}\subset R_{\tii}$ satisfies the Ore condition.
	\end{cor}

	\begin{proof}
		This can be seen from the lemma above and the $q$-commutative relation
		\begin{equation}\label{eq-qcomm}
			[x^{\aa}y^{\bb},x^{\aa'}y^{\bb'}]_q=\aa\tr D_{\tii}\bb'-\aa^{\prime\mathrm{T}}D_{\tii}\bb,\aa,\aa',\bb,\bb'\in\Z^m,
		\end{equation}
		where $\aa,\aa',\bb,\bb'$ are viewed as column vectors.
	\end{proof}

	Denote $\ttii=R_{\tii}[S_{\tii}\inv]$. Since $E_{\ww}\subset \cqg/J_{\ww}\simeq R_{\tii}$ is contained in $S_{\tii}$ (Proposition \ref{prop-tpinm} when $j=0$), $\ttii$ is also a localization of $\cqg_{\ww}$. We determine its structure using quasi-$\tii$-normal elements. Denote $\pp_{\tii}=\sqcup_{(\nu,\mu)\in P\times P}\pp_{\tii,\nu,\mu}$ the set of weight strings of type $\tii$.

	\begin{defn}
		In view of Lemma \ref{lem-ore}, we call two weight strings $\muu=(\mu_0,\mu_1,\cdots,\mu_m)$, $\muu'=(\mu_0',\mu_1',\cdots,\mu_m')\in\pp_{\tii}$ are \emph{$\tii$-separable} if there exists some $x^{\aa}y^{\bb}\in S_{\tii}$ such that $[x^{\aa}y^{\bb},I(\muu)]_q\ne[x^{\aa}y^{\bb},I(\muu')]_q$.
		
		Given $1\le k\le m$, we call $\muu,\muu'$ are \emph{$k^+$-different} if $\mu_i=\mu_i',i<k,\mu_k\ne\mu_k'$, \emph{$k^-$-different} if $\mu_i=\mu_i',i\ge k,\mu_{k-1}\ne\mu_{k-1}'$.
	\end{defn}

	\begin{prop}\label{prop-sielt}
		Elements of $S_{\tii}$ are nonzero scalar multiples of $I(\muu)$'s in $R_{\tii}$, $\muu\in\pp_{\tii}$.
	\end{prop}

	\begin{proof}
		Proposition \ref{prop-pii} tells that elements of $R_{\tii}$ are linear combinations of $I(\muu)$'s, $\muu\in\pp_{\tii}$. The proposition then follows from Proposition \ref{prop-qtunit}.
	\end{proof}

	\begin{prop}\label{prop-separable}
		For $1\le k\le m$, any two $k^\pm$-different weight strings $\muu,\muu'\in\pp_{\tii}$ are $\tii$-separable.
	\end{prop}

	\begin{proof}
		We prove the $k^+$-different case. Write $\muu=(\mu_0,\mu_1,\cdots,\mu_m)$, $\muu'=(\mu_0',\mu_1',\cdots,\mu_m')$. Use induction on $k$. When $k=1$, $\mu_0=\mu_0'$ and $\mu_1\ne\mu_1'$. Suppose $\mu_1=\mu_0-j_1\sgn(i_1)\alpha_{|i_1|},\mu_1'=\mu_0-j_1'\sgn(i_1)\alpha_{|i_1|},j_1,j_1'\ge0,j_1\ne j_1'$. Choose $\lambda_1\in P^+$ such that $(w_{>1}^{-\sgn(i_1)}(\lambda_1),\alpha_{|i_1|}^\vee)>0$, i.e. $\nl{\tii}{\lambda_1}(1)$ is well defined (c.f. (\ref{eq-quasinormL})). A direct calculation using Proposition \ref{prop-tpinm} shows
		\[d_{\tii,\lambda_1}^l(1):=\tpi_{\tii}(\nl{\tii}{\lambda_1}(1))\inv\tpi_{\tii}(\nl{\tii}{\lambda_1}(0))=Cx^{\sgn(i_1)}y\otimes1\cdots\otimes1,\exists C\in\C^*,\]
		thus
		\[[d_{\tii,\lambda_1}^l(1),I(\muu)]_q-[d_{\tii,\lambda_1}^l(1),I(\muu')]_q=[d_{\tii,\lambda_1}^l(1),I(\muu)/I(\muu')]_q=(j_1-j_1')(\alpha_{|i_1|},\alpha_{|i_1|})\ne0.\]
		So one of $\tpi_{\tii}(\nl{\tii}{\lambda_1}(1))$ and $\tpi_{\tii}(\nl{\tii}{\lambda_1}(0))$ must have different $q$-commute indices with $I(\muu)$ and $I(\muu')$, i.e. $\muu$ and $\muu'$ are $\tii$-separable.

		\sloppy Suppose any two $i^+$-different weight strings are $\tii$-separable, $1\le i\le k-1$. Let $\muu,\muu'$ be $k^+$-different, i.e. $\mu_i=\mu_i',0\le i\le k-1,\mu_k=\mu_{k-1}-j_k\sgn(i_k)\alpha_{|i_k|},\mu_k'=\mu_{k-1}-j_k'\sgn(i_k)\alpha_{|i_k|},j_k,j_k'\ge0,j_k\ne j_k'$. We want to show $\muu,\muu'$ are $\tii$-separable. Denote $\tii_{\ge k}=(\ti_k,\ti_{k+1},\cdots,\ti_m)$. Choose $\lambda_k\in P^+$ such that $\nl{\tii_{\ge k}}{\lambda_k}(1)$ is well defined, i.e. $(w_{>k}^{-\sgn(i_k)}(\lambda_k),\alpha_{|i_k|}^\vee)>0$. For $j=0,1$, define weight strings $\nuu^l_{\tii,\lambda_k}(k,j)=(\nu_{j0},\nu_{j1},\cdots,\nu_{jm})\in\pp_{\tii}$ such that $(\nu_{j(k-1)},\nu_{jk},\cdots,\nu_{jm})=\muu_{\tii_{\ge k},\lambda_k}^l(j)$ (c.f. Proposition \ref{prop-tpinm}) and $\nu_{j0}=\nu_{j1}=\cdots=\nu_{j(k-1)}$.
		
		Apply Proposition \ref{prop-pii} (ii) to compute $\tpi_{\tii}(\nl{\tii_{\ge k}}{\lambda_k}(j)),j=0,1$. There is only one nonzero summand term corresponding to a weight string $\nuu=(\nu_0,\nu_1,\cdots,\nu_m)\in\pp_{\tii}$ such that $\nu_0=\nu_1=\cdots=\nu_{k-1}=\nu_{j(k-1)}$, i.e. such $\nuu$ must be $\nuu^l_{\tii,\lambda_k}(k,j)$, c.f. Remark \ref{rem-pii} and Corollary \ref{cor-DemMod}. This shows weight strings of other nonzero summand terms must be $i^+$-different with $\nuu^l_{\tii,\lambda_k}(k,j)$, for some $1\le i\le k-1$. In view of Lemma \ref{lem-ore}, the induction hypothesis indicates that $I(\nuu^l_{\tii,\lambda_k}(k,j))\in \ttii=R_{\tii}[S_{\tii}\inv],j=0,1$. Denote $\muu_{\ge k-1}=(\mu_{k-1},\mu_k,\cdots,\mu_m),\muu_{\ge k-1}'=(\mu_{k-1}',\mu_k',\cdots,\mu_m')$ and
		\[d_{\tii,\lambda_k}^l(k):=I(\nuu^l_{\tii,\lambda_k}(k,1))\inv I(\nuu^l_{\tii,\lambda_k}(k,0))\in \lqi^\times.\]
		Notice $\muu,\muu'$ are $k^+$-different, calculate using the $k=1$ case
		\[[d_{\tii,\lambda_k}^l(k),I(\muu)]_q-[d_{\tii,\lambda_k}^l(k),I(\muu')]_q=[d_{\tii,\lambda_k}^l(k),I(\muu)/I(\muu')]_q\]
		\[=[d_{\tii_{\ge k},\lambda_k}^l(1),I(\muu_{\ge k-1})/I(\muu'_{\ge k-1})]_q=(j_k-j_k')(\alpha_{|j_k|},\alpha_{|j_k|})\ne0.\]
		Therefore one of $I(\nuu^l_{\tii,\lambda_k}(k,0))$ and $I(\nuu^l_{\tii,\lambda_k}(k,1))$ has different $q$-commute indices with $I(\muu)$ and $I(\muu')$, so does one of their numerators and denominators. Note the numerators of $I(\nuu^l_{\tii,\lambda_k}(k,j))\in\ttii$ fall in $\lqi^\times\cap R_{\tii}=S_{\tii},j=0,1$, as well as the denominators. This implies some element in $S_{\tii}$ has different $q$-commute indices with $I(\muu)$ and $I(\muu')$, i.e. $\muu$ and $\muu'$ are $\tii$-separable.

		For the $k^-$-different case, use downward induction on $k$. Choose $\lambda_k'\in P^+$ such that $(\lambda_k',\alpha_{|i_k|}^\vee)>0$. For $j=0,1$, define weight strings $\nuu^r_{\tii,\lambda_k'}(k,j)=(\nu_{j0},\nu_{j1},\cdots,\nu_{jm})\in\pp_{\tii}$ such that $(\nu_{j0},\nu_{j1},\cdots,\nu_{jk})=\muu_{\tii_{\le k},\lambda_k'}^r(j)$ and $\nu_{jk}=\nu_{j(k+1)}=\cdots=\nu_{jm}$. Denote $d_{\tii,\lambda_k'}^r(k):=I(\nuu^r_{\tii,\lambda_k'}(k,1))\inv I(\nuu^r_{\tii,\lambda_k'}(k,0))\in \lqi^\times$. The rest of proof is similar to that of the $k^+$-different case.
	\end{proof}

	\begin{rem}\label{rem-separable}
		By induction on $k$, the proof above shows for $k^+$-different weight strings $\muu,\muu'$, $1\le k\le m$, there exists some element in $S_{\tii}$ of the form
		\[\prod _{j=1}^k I(\nuu^l_{\tii,\lambda_j}(j,0))^{h_j}I(\nuu^l_{\tii,\lambda_j}(j,1))^{h_j'},h_j,h_j'\ge0,\lambda_j\in P^+,\]
		which has different $q$-commute indices with $I(\muu)$ and $I(\muu')$. Similarly, for $k^-$-different weight strings $\muu,\muu'$, $1\le k\le m$, there exists some element in $S_{\tii}$ of the form
		\[\prod _{j=k}^m I(\nuu^r_{\tii,\lambda_j'}(j,0))^{k_j}I(\nuu^r_{\tii,\lambda_j'}(j,1))^{k_j'},k_j,k_j'\ge0,\lambda_j'\in P^+,\]
		which has different $q$-commute indices with $I(\muu)$ and $I(\muu')$.
	\end{rem}

	\begin{cor}\label{prop-accessible}
		$I(\muu)^{\pm 1}\in\ttii,\forall\muu\in\pp_{\tii}$.
	\end{cor}

	\begin{proof}
		First we notice that $I(\muu_1+\muu_2)=I(\muu_1)I(\muu_2)$ for $\muu_1,\muu_2\in\pp_{\tii}$. Write $\muu=(\mu_0,\mu_1,\cdots,\mu_m),\mu_{k-1}-\mu_k=j_k\sgn(i_k)\alpha_{|i_k|},j_k\ge 0,1\le k\le m$ and $\mu_0=l_1\omega_1+l_2\omega_2+\cdots+l_n\omega_n$. $\muu$ can be decomposed into
		\[\muu=\sum_{i=1}^n l_i\omegaa_{i}+\sum_{k=1}^m j_k\muu_{\tii,k},\]
		where
		\begin{equation}\label{eq-omegamu}
			\begin{aligned}
				\omegaa_{i}&=(\omega_i,\omega_i,\cdots,\omega_i)\in\pp_{\tii},1\le i\le n,\\
				\muu_{\tii,k}&=(\underbrace{0,\cdots,0}_{k\text{ times}},-\sgn(i_k)\alpha_{|i_k|},\cdots,-\sgn(i_k)\alpha_{|i_k|})\in\pp_{\tii},1\le k\le m.
			\end{aligned}
		\end{equation}
		Thus
		\[I(\muu)=\prod_{i=1}^n I(\omegaa_{i})^{l_i}\prod_{k=1}^m I(\muu_{\tii,k})^{j_k}.\]
		It suffices to prove $I(\omegaa_{i})^{\pm 1}\in\ttii,1\le i\le n,I(\muu_{\tii,k})^{\pm 1}\in\ttii,1\le k\le m$. Calculate $\tpi_{\tii}(\np{e}{\omega_i})$ using Proposition \ref{prop-tpinm},
		\[\tpi_{\tii}(\np{e}{\omega_i})=c_{\omegaa_i}I(\omegaa_i)+\sum_{\muu'}c_{\muu'}I(\muu'),c_{\omegaa_i}\ne0,\]
		where the summation is over $\muu'\in\pp_{\tii,\omega_i,\omega_i}-\{\omegaa_i\}$, so $\muu'$ is $k^\pm$-different from $\omegaa_i$ for some $1\le k\le m$. Proposition \ref{prop-separable} then shows $\muu'$ is $\tii$-separable with $\omegaa_i$. By Lemma \ref{lem-ore}, $I(\omegaa_i)\in\ttii$. The numerator and denominator of $I(\omegaa_i)\in\ttii=R_{\tii}[S_{\tii}\inv]$ are all in $S_{\tii}$, thus $I(\omegaa_i)\inv\in\ttii$.

		Denote $\muu_{\tii,k}'=\sgn(i_k)\omegaa_{|i_k|}+\muu_{\tii,k}\in\pp_{\tii}$. Calculate $\tpi_{\tii}(c^{\sgn(i_k)}_{s_{|i_k|},\omega_i})$ using Proposition \ref{prop-tpinm},
		\[\tpi_{\tii}(c^{\sgn(i_k)}_{s_{|i_k|},\omega_i})=c'_{\muu_{\tii,k}'}I(\muu_{\tii,k}')+\sum_{\muu'}c'_{\muu'}I(\muu'),c'_{\muu_{\tii,k}'}\ne0,\]
		where the summation is over $\muu'\in\pp_{\tii,\sgn(i_k)\omega_{|i_k|},\sgn(i_k)s_{|i_k|}(\omega_{|i_k|})}-\{\muu_{\tii,k}'\}$. Similarly, by Proposition \ref{prop-separable} and Lemma \ref{lem-ore}, we have $I(\muu_{\tii,k}')=I(\omegaa_{|i_k|})^{\sgn(i_k)}I(\muu_{\tii,k})\in\ttii$. Thus $I(\muu_{\tii,k})^{\pm 1}\in\ttii$.
	\end{proof}

	We are ready to state the main theorem of this subsection.

	\begin{thm}\label{thm-localization}
		$\ttii$ is a quantum torus with $q$-commute generators $I(\muu)^{\pm1},\muu\in\pp_{\tii}$.
	\end{thm}

	\begin{proof}
		Corollary \ref{prop-accessible} shows $\ttii$ contains the quantum torus generated by $I(\muu)^{\pm1},\muu\in\pp_{\tii}$. On the other hand, Proposition $\ref{prop-pii}$ tells images of matrix coefficients under $\tpi$ are in the subalgebra generated by $I(\muu),\muu\in\pp_{\tii}$. Combined with Proposition \ref{prop-sielt}, $\ttii$ is contained in the subalgebra generated by $I(\muu)^{\pm1},\muu\in\pp_{\tii}$.
	\end{proof}

	\subsection{Weight Spaces of $\tmii$}

	In this subsection we make further observations of $\ttii$ and $\tmii$ as $\ttii$-module. First we determine $n(\tii)$, the dimension of $\ttii$ as quantum torus. Theorem \ref{thm-localization} tells that $\ttii$ is generated by $I(\muu)^{\pm1}=x^{\pm\aa(\muu)}y^{\pm\bb(\muu)},\muu\in\pp_{\tii}$. Denote by $\Z_{\tii}=\bigoplus_{\muu\in\pp_{\tii}}\Z\,\aa(\muu)\oplus\bb(\muu)\subset \Z^{2m}$ the group of exponent vectors of elements in $\ttiit$, where $\oplus$ stands for concatenation. $n(\tii)$ is then the rank of $\Z_{\tii}$.
	
	Recall weight strings $\omegaa_{i},1\le i\le n$ and $\muu_{\tii,k},1\le k\le m$ defined by (\ref{eq-omegamu}). $\Z_{\tii}$ is then generated by $\aa(\omegaa_{i})\oplus\bb(\omegaa_{i}),1\le i\le n$ and $\aa(\muu_{\tii,k})\oplus\bb(\muu_{\tii,k}),1\le k\le m$. Put these column vectors together to obtain a $2m\times(n+m)$ matrix $\Phi(\tii)$
	\begin{equation}\label{eq-matTii}
		\begin{gathered}
			\Phi(\tii):=\begin{pmatrix}
				\Omega(\tii)_{m\times n} & \Lambda(\tii)_{m\times m}\\
				0_{m\times n} & I_m
			\end{pmatrix},\\
			\Omega(\tii)=(\aa(\omegaa_{1}),\aa(\omegaa_{2}),\cdots,\aa(\omegaa_{n})),\Omega(\tii)_{st}=(\alpha_{|i_s|}^\vee,\omega_t),\\
			\Lambda(\tii)=(\aa(\muu_{\tii,1}),\aa(\muu_{\tii,2}),\cdots,\aa(\muu_{\tii,m})),\Lambda(\tii)_{st}=\begin{cases}
				0, & s<t,\\
				-\sgn(i_t), & s=t,\\
				-\sgn(i_t)(\alpha_{|i_s|}^\vee,\alpha_{|i_t|}), & s>t.
			\end{cases}
		\end{gathered}
	\end{equation}
	Therefore
	\begin{equation}\label{eq-nii}
		n(\tii)=\rr(\Phi(\tii))=m+\rr(\Omega(\tii))=l(w_1)+l(w_2)+|\supp(\ww)|.
	\end{equation}

	\sloppy Next we determine $d(\tii)$, the dimension of $\ttii$'s center $Z(\ttii)$ as Laurent polynomial algebra, c.f. Corollary \ref{cor-qtcenter}. Use (\ref{eq-qcomm}) and (\ref{eq-matTii}) to calculate $q$-commute indices between $x^{\aa(\omegaa_{i})}y^{\bb(\omegaa_{i})},1\le i\le n$ and $x^{\aa(\muu_{\tii,k})}y^{\bb(\muu_{\tii,k})},1\le k\le m$, one gets the $(n+m)\times(n+m)$ skew symmetric matrix $H(\tii)$ for $\ttii$
	\[H(\tii):=\begin{pmatrix}
		0_{n\times n} & \Omega(\tii)\tr D_{\tii}\\
		-D_{\tii}\Omega(\tii) & \Lambda(\tii)\tr D_{\tii}-D_{\tii}\Lambda(\tii)
	\end{pmatrix}.\]
	We claim $\rr(H(\tii))=m+\rr(w_1-w_2)$, and leave the tedious calcuation to Section \ref{sec-appen1}. Thus
	\begin{equation}\label{eq-dii}
		d(\tii)=m+n-\rr(\hh(\tii))=n-\rr(w_1-w_2)=\dim\ker(w_1-w_2).
	\end{equation}

	We will illustrate some applications of the calcuations above.

	\begin{prop}\label{prop-prim}
		$Z_{\ww}=Z(\ttii)$. Therefore $\prim\cqg_{\ww}$ is homeomorphic to the maximal spectrum of $Z(\ttii)$.
	\end{prop}

	\begin{proof}
		Theorem \ref{thm-genzww} tells that $Z_{\ww}$ has generators in $E_{\ww}E_{\ww}\inv$. Denote $\zew=\{\aa\oplus\bb\in\Z^{2m}|x^{\aa}y^{\bb}\in \C^*E_{\ww}E_{\ww}\inv\}$. Since $Z_{\ww}$ and $Z(\ttii)$ have the same dimension $d(\tii)$ as Laurent polynomial algebra (c.f. Theorem \ref{thm-stratification}), it suffices to show $\zew$ is a direct summand of $\Z_{\tii}$, which implies the generators of $Z_{\ww}$ in Theorem \ref{thm-genzww} is part of algebracally independent generators of $\ttii$ (as quantum torus), thus generate $Z(\ttii)$. $\zew$ is generated by $\aa(\muu_{\tii,\lambda}^\pm)\oplus\bb(\muu_{\tii,\lambda}^\pm),\lambda\in P^+$, where $\muu_{\tii,\lambda}^\pm$ are weight strings
		\[\muu_{\tii,\lambda}^-=(w_1(\lambda),\wm{1}(\lambda),\cdots,\wm{m-1}(\lambda),\lambda),\]
		\[\muu_{\tii,\lambda}^+=(-w_2(\lambda),-\wp{1}(\lambda),\cdots,-\wp{m-1}(\lambda),-\lambda).\]
		Write
		\[\muu_{\tii,\lambda}^-=\sum_{i=1}^n(w_1(\lambda),\alpha_i^\vee)\omegaa_i+\sum_{k=1}^m\chi_k^-(\wm{k}(\lambda),\alpha_{|i_k|}^\vee)\muu_{\tii,k},\]
		\[\muu_{\tii,\lambda}^+=\sum_{i=1}^n-(w_2(\lambda),\alpha_i^\vee)\omegaa_i+\sum_{k=1}^m\chi_k^+(\wp{k}(\lambda),\alpha_{|i_k|}^\vee)\muu_{\tii,k},\]
		where $\chi_k^\pm=(1\pm\sgn(i_k))/2$. It suffices to show the matrix
		\[\Psi(\tii):=\begin{pmatrix}
			\big((w_1(\omega_s),\alpha_t^\vee)\big)_{1\le s,t\le n} & \big(\chi_t^-(\wm{t}(\omega_s),\alpha_{|i_t|}^\vee)\big)_{1\le s\le n,1\le t\le m}\\
			-\big((w_2(\omega_s),\alpha_t^\vee)\big)_{1\le s,t\le n} & \big(\chi_t^+(\wp{t}(\omega_s),\alpha_{|i_t|}^\vee)\big)_{1\le s\le n,1\le t\le m}
		\end{pmatrix}\]
		has determinantal divisors 1 except 0. For $w\in W$, denote $P(w)\in GL_n(\Z)$ such that
		\[(w(\omega_1),w(\omega_2),\cdots,w(\omega_n))=(\omega_1,\omega_2,\cdots,\omega_n)P(w).\]
		Then
		\[\begin{pmatrix}
			I_n & \\
			I_n & I_n
		\end{pmatrix}
		\begin{pmatrix}
			P(w_1\inv)\tr & \\
			& P(w_2\inv)\tr
		\end{pmatrix}
		\Psi(\tii)=\begin{pmatrix}
			I_n & \big(\chi_t^-(\omega_s,w_{\le t}^-(\alpha_{|i_t|}^\vee))\big)_{1\le s\le n,1\le t\le m}\\
			0 & \big((\omega_s,w_{\le t}^{\sgn(i_t)}(\alpha_{|i_t|}^\vee))\big)_{1\le s\le n,1\le t\le m}
		\end{pmatrix}.\]
		So it suffices to prove $\big((\omega_s,w_{\le t}^{\sgn(i_t)}(\alpha_{|i_t|}^\vee))\big)_{1\le s\le n,1\le t\le m}$ has determinantal divisors 1 except 0. This is true since $\{w_{\le t}^{\sgn(i_t)}(\alpha_{|i_t|}^\vee)\}_{1\le t\le m}$ is $\Z$-equivalent to $\{\alpha_{|i_t|}^\vee\}_{1\le t\le m}$. 

		Now we have proved $Z_{\ww}=Z(\ttii)$. Therefore by Theorem \ref{thm-stratification} (iii), $\prim\cqg_{\ww}\simeq\specm Z_{\ww}=\specm Z(\ttii)$ as homeomorphism.
	\end{proof}
	
	For $\nn=\nnm\tr\in\Z^m$, denote $\te_{\nn}=e_{n_1}\otimes e_{n_2}\otimes\cdots\otimes e_{n_m}\in\tmii$. Such $\te_{\nn}$'s form a basis of $\tmii$. We have
	\begin{equation}\label{eq-xyaction}
		x^{\aa}\te_{\nn}=\te_{\nn-\aa},\,y^{\bb}\te_{\nn}=\gamma_1\gamma_2\cdots\gamma_mq^{\bb\tr D_{\tii}\nn}\te_{\nn},\aa,\bb\in\Z^m.
	\end{equation}
	Therefore $\tmii$ is a $\langle y^{\bb},\bb\in\Z^m\rangle$-weight module. Let $\ttiz=\ttii\cap\langle y^{\bb},\bb\in\Z^m\rangle$. Notice in (\ref{eq-matTii}), the block $\Lambda(\tii)$ has determinant $\pm1$, thus the column vectors of
	\begin{equation}\label{eq-tPhi}
		\tPhi(\tii):=
		\begin{pmatrix}
			0_{m\times n} & I_m \\
			\Lambda(\tii)\inv\Omega(\tii) & \Lambda(\tii)\inv
		\end{pmatrix}
	\end{equation}
	are $\Z$-equivalent to those of $\Phi(\tii)$ and generate $\Z_{\tii}$. Denote
	\[\tOmega(\tii)=\Lambda(\tii)\inv\Omega(\tii)=(\bb_{\tii,1},\bb_{\tii,2},\cdots,\bb_{\tii,n}).\]
	Then $\ttiz$ is a Laurent polynomial algebra generated by $y^{\bb_{\tii,1}},y^{\bb_{\tii,2}},\cdots,y^{\bb_{\tii,n}}$ of dimension $s(\tii):=\rr(\tOmega(\tii))=|\supp(\ww)|$. Let $\Theta(\tii)=(\ccc_{\tii,1},\ccc_{\tii,2},\cdots,\ccc_{\tii,s(\tii)})$, where $\ccc_{\tii,1},\ccc_{\tii,2},\cdots,\ccc_{\tii,s(\tii)}$ form a basis of the $\Z$-span of column vectors of $\tOmega(\tii)\tr D_{\tii}$. $\ttiz$-weight spaces of $\tmii$ are of the form
	\begin{equation}\label{eq-wtspace}
		\tmii(\mm):=\bigoplus_{\tOmega(\tii)\tr D_{\tii}\nn=\Theta(\tii)\mm}\C\te_{\nn},\mm\in\Z^{s(\tii)},
	\end{equation}
	with scalar $\ttiz$-actions on it. Let $\cc_{\tii}=C_{\ttii}(\ttiz)$ be the centralizer of $\ttiz$ in $\ttii$. Any $\tmii(\mm),\mm\in\Z^{s(\tii)}$ is a $\cc_{\tii}$-module. Denote $\S(\tmii)$ the category of $\ttii$-submodules of $\tmii$, $\S(\tmii(\mm))$ the category of $\cc_{\tii}$-submodules of $\tmii(\mm)$. We show they are equivalent categories in the next proposition.

	\begin{rem}\label{rem-xaa}
		For any $\aa\in\Z^m$, (\ref{eq-tPhi}) shows we can choose some $\bb'\in\Z^m$, such that $x^{\aa}y^{\bb'}\in\ttii$. Denote this element by $x(\aa)\in\ttii$. Notice that elements of $\ttiit$ are of the form $x(\aa)\tt,\aa\in\Z^m,\tt\in\ttizt$, and for any $\aa_1,\aa_2\in\Z^m,x(\aa_1)x(\aa_2)\in x(\aa_1+\aa_2)\ttizt$.
	\end{rem}

	\begin{prop}\label{prop-equivcat}
		$\S(\tmii)$ and $\S(\tmii(\mm))$ are equivalent categories via the functors
		\[\varphi_{\tii,\mm}:\S(\tmii)\to\S(\tmii(\mm)),\varphi_{\tii,\mm}(N)=N\cap\tmii(\mm),N\in\S(\tmii),\]
		\[\psi_{\tii,\mm}:\S(\tmii(\mm))\to\S(\tmii),\psi_{\tii,\mm}(N_0))=\ttii N_0,N_0\in\S(\tmii(\mm)).\]
	\end{prop}

	\begin{proof}
		For any $N\in\S(\tmii)$, denote $N(\mm')=N\cap\tmii(\mm'),\mm'\in\Z^{s(\tii)}$. A Vandermonde determinant argument shows
		\[N=\bigoplus_{\mm'\in\Z^{s(\tii)}}N(\mm').\]
		Besides, for any $\mm_1,\mm_2\in\Z^{s(\tii)}$, there exists some $\nn\in\Z^m$ such that $\tOmega(\tii)\tr D_{\tii}\nn=\Theta(\tii)(\mm_1-\mm_2)$, thus
		\[x(\nn)N(\mm_1)\subset N(\mm_2),x(-\nn)N(\mm_2)\subset N(\mm_1),\]
		which leads to $x(\nn)N(\mm_1)=N(\mm_2)$ since $x(\nn)x(-\nn)\in\ttizt$ acts as nonzero scalars on weight spaces, c.f. (\ref{eq-xyaction}) and (\ref{eq-wtspace}). Therefore $N=\ttii N(\mm)$, i.e. $\psi_{\tii,\mm}\varphi_{\tii,\mm}=\id_{\S(\tmii)}$. On the other hand, for any $N_0\in\S(\tmii(\mm))$, (\ref{eq-wtspace}) implies
		\[\ttii N_0\cap\tmii(\mm)=\cc_{\tii}N_0=N_0,\]
		this shows $\varphi_{\tii,\mm}\psi_{\tii,\mm}=\id_{\S(\tmii(\mm))}$.
	\end{proof}

	The proposition above suggests the $\ttii$-module structure of $\tmii$ is closely related to the $\cc_{\tii}$-module structure of $\tmii(\mm),\mm\in\Z^{s(\tii)}$, which in turn is related to the center $Z(\cc_{\tii})$ of $\cc_{\tii}$.

	\begin{prop}\label{prop-ccenter}
		$Z(\cc_{\tii})=\ttiz Z(\ttii)$.
	\end{prop}

	\begin{proof}
		Obviously $Z(\cc_{\tii})\supset\ttiz Z(\ttii)$. Define the skew symmetric bilinear form $g_{\tii}(\cdot,\cdot)$ on $\C^{2m}\times\C^{2m}$ by
		\[g_{\tii}(\aa\oplus\bb,\aa'\oplus\bb')=\aa\tr D_{\tii}\bb'-\aa^{\prime\mathrm{T}}D_{\tii}\bb,\aa,\aa',\bb,\bb'\in\C^m.\]
		Thus ${g_{\tii}(\aa\oplus\bb,\aa'\oplus\bb')}=[x^{\aa}y^{\bb},x^{\aa'}y^{\bb'}]_q,\aa,\aa',\bb,\bb'\in\Z^m$. Denote by $V,V_1,V_2$ the column spaces of $\tPhi(\tii),(0_{n\times m},\tOmega(\tii)\tr)\tr,(I_m,(\Lambda(\tii)\inv)\tr)\tr$ respectively, c.f. (\ref{eq-tPhi}). Use symbol $\perp$ for the orthogonal complement under $g_{\tii}$. It follows that $V=V_1\oplus V_2,g_{\tii}|_{V_1\times V_1}=0$ and $V_1\cap V_2^\perp=0$. So we can write $V_2=V_2'\oplus V_2''$ such that $g_{\tii}(V_1,V_2')=0$ and $g|_{V_1\times V_2''}$ is a perfect pairing. Then $V_1$ and $V_3:=V_1\oplus V_2'$ are spanned by exponent vectors of elements in $\ttizt$ and $\cc_{\tii}=C_{\ttii}(\ttiz)$ respectively. Besides, $V_3\cap V_3^\perp=V_1\oplus(V_2'\cap (V_2')^\perp)$ is spanned by exponent vectors of elements in $Z(\cc_{\tii})^\times$. Notice for any $v\in V_2'\cap(V_2')^\perp$, there exists some $v'\in V_1$ such that $g_{\tii}(v-v',V_2'')=0$, thus $g_{\tii}(v-v',V)=0$ and $v-v'\in V^\perp$. This implies $V_3\cap V_3^\perp\subset V_1+V^\perp$ and therefore $Z(\cc_{\tii})\subset\ttiz Z(\ttii)$.
	\end{proof}

	In view of (\ref{eq-xyaction}),(\ref{eq-wtspace}) and Corollary \ref{cor-qtcenter}, the above proposition indicates

	\begin{cor}\label{cor-ccdecomp}
		Let $k(\tii)=(m-d(\tii)-s(\tii))/2$. There exist nonnegative integers $m_{\tii,1}\le m_{\tii,2}\le\cdots\le m_{\tii,k(\tii)}$ such that $\cc_{\tii}\simeq \ttiz\otimes\cc_{\tii}'$, where
		\[\cc_{\tii}'=Z(\ttii)\otimes L_{q^{m_{\tii,1}}}(2)\otimes L_{q^{m_{\tii,2}}}(2)\otimes\cdots\otimes L_{q^{m_{\tii,k(\tii)}}}(2)\]
		is a quantum torus of dimension $m-s(\tii)$ and $\tmii(\mm)\simeq\cc_{\tii}'$ as (left) $\cc_{\tii}'$-module, $\mm\in\Z^{s(\tii)}$.
	\end{cor}

	We want to construct simple $\ttii$-modules through quotient modules of $\tmii$. The discussion above implies maximal $\ttii$-submodules of $\tmii$ are in bijection with maximal left ideals of $\cc_{\tii}'$. Therefore we have proved the following theorem.

	\begin{thm}\label{thm-simquot}
		$\tmii$ as $\ttii$-module is indecomposable, and its simple quotient modules are in bijection with simple $\cc_{\tii}'$-modules.
	\end{thm}

	In light of (\ref{eq-tirr}) and the theorem above, construct
	\begin{equation}\label{eq-tirrti}
		\tirr\ttii:=\tirr\cc_{\tii}'=\T^{d(\tii)}\times\irr L_{q^{m_{\tii,1}}}(2)\times\irr L_{q^{m_{\tii,2}}}(2)\times\cdots\times\irr L_{q^{m_{\tii,k(\tii)}}}(2)
	\end{equation}
	to be a bundle of simple $\ttii$-modules over $\T^{d(\tii)}\simeq\specm Z(\ttii)=\specm Z_{\ww}$.

	\subsection{Calculation of $\rr(H(\tii)$)}\label{sec-appen1}

	\sloppy We need the following notations. Given $\jj=(j_1,j_2,\cdots,j_{l(\jj)})\in\I_w,w\in W$, let
	\[\beta_{\jj,k}=s_{j_1}s_{j_2}\cdots s_{j_{k-1}}(\alpha_{j_k}),1\le k\le l(\jj).\]
	Define $l(\jj)\times l(\jj)$ matrices $B(\jj),\tb(\jj),A(\jj),\ta(\jj),D(\jj),P(\jj),\tp(\jj)$ and $n\times l(\jj)$ matrices $C(\jj),\tc(\jj)$ to be
	\[B(\jj)_{st}=(\beta_{\jj,s}^\vee,\beta_{\jj,t}),\tb(\jj)_{st}=(\alpha_{j_s}^\vee,\alpha_{j_t}),s<t;B(\jj)_{st}=\tb(\jj)_{st}=0,s\ge t;\]
	\[A(\jj)=B(\jj)-B(\jj)\tr,\ta(\jj)=\tb(\jj)-\tb(\jj)\tr,\]
	\[C(\jj)_{st}=(\omega_s,\beta_{\jj,t}),\tc(\jj)_{st}=(\omega_s,\alpha_{j_t}),\]
	\[D(\jj)=\frac{1}{2}\diag((\alpha_{j_1},\alpha_{j_1}),(\alpha_{j_2},\alpha_{j_2}),\cdots,(\alpha_{j_{l(\jj)}},\alpha_{j_{l(\jj)}}))\]
	\[=\frac{1}{2}\diag((\beta_{\jj,1},\beta_{\jj,1}),(\beta_{\jj,2},\beta_{\jj,2}),\cdots,(\beta_{\jj,l(\jj)},\beta_{\jj,l(\jj)})),\]
	\[P(\jj)=I_{l(\jj)}+D(\jj)\inv B(\jj),\tp(\jj)=I_{l(\jj)}+D(\jj)\inv\tb(\jj).\]
	If given another $\jj'=(j_1',j_2',\cdots,j_{l(\jj)}')\in\I_{w'},w'\in W$, define $l(\jj)\times l(\jj')$ matrix $\tc(\jj,\jj')$ to be $\tc(\jj,\jj')_{st}=(\alpha_{j_s},\alpha_{j_t'})$.
	
	Exchanging rows and columns of $H(\tii)$ by considering $q$-commute indices between reordered generators of $\ttii$,
	\[x^{\aa(\muu_{\tii,k})}y^{\bb(\muu_{\tii,k})},\ti_k>0;x^{\aa(\muu_{\tii,k})}y^{\bb(\muu_{\tii,k})},\ti_k<0;x^{\aa(\omegaa_{i})}y^{\bb(\omegaa_{i})},1\le i\le n,\]
	one gets an $(m+n)\times(m+n)$ skew symmetric matrix which is congruent with $H(\tii)$,
	\[\begin{pmatrix}
		-\ta(\tii^+) & -\tc(\tii^+,\tii^-) & -\tc(\tii^+)\tr\\
		\tc(\tii^+,\tii^-)\tr & \ta(\tii^-) & -\tc(\tii^-)\tr\\
		\tc(\tii^+) & \tc(\tii^-) & 0
	\end{pmatrix},\]
	where $\tii^+$ (resp. $\tii^-$) stands for reduced expression of $w_2$ (resp. $w_1$) extracted from $\tii$. This matrix can be simplified by eliminating the blocks $-\tc(\tii^+,\tii^-),\tc(\tii^+,\tii^-)\tr$ using rows and columns of blocks $\tc(\tii^-)$ and $-\tc(\tii^-)\tr$. After changing signs of the blocks $\tc(\tii^-)$ and $-\tc(\tii^-)\tr$, one obtains a matrix $\th(\tii)$ which is congruent with $H(\tii)$,
	\[\th(\tii):=\begin{pmatrix}
		-\ta(\tii^+) & 0 & -\tc(\tii^+)\tr\\
		0 & \ta(\tii^-) & \tc(\tii^-)\tr\\
		\tc(\tii^+) & -\tc(\tii^-) & 0
	\end{pmatrix}.\]

	Define $\hh(\tii)$ to be the matrix
	\[\hh(\tii):=\begin{pmatrix}
		A(\tii^+) & 0 & -C(\tii^+)\tr\\
		0 & -A(\tii^-) & C(\tii^-)\tr\\
		C(\tii^+) & -C(\tii^-) & 0
	\end{pmatrix}.\]
	De. Concini and Procesi proved $\rr(\hh(\tii))=l(w_1)+l(w_2)+\rr(w_1-w_2)$ in \cite{DP}*{Proposition 4.10}, see also \cite{DKP}*{Section 3.2-3.4}. We claim that $\th(\tii)$ and $\hh(\tii)$ are congruent, thus $\rr(H(\tii))=\rr(\hh(\tii))$.	We need the following lemma.
	\begin{lem}\label{lem-alphabeta}
		(i) $(\beta_{\jj,1},\beta_{\jj,2},\cdots,\beta_{\jj,l(\jj)})=(\alpha_{j_1},\alpha_{j_2},\cdots,\alpha_{j_{l(\jj)}})\tp(\jj)$.

		(ii) $(\alpha_{j_1},\alpha_{j_2},\cdots,\alpha_{j_{l(\jj)}})=(\beta_{\jj,1},\beta_{\jj,2},\cdots,\beta_{\jj,l(\jj)})P(\jj)$.

		(iii) $\tp(\jj)P(\jj)=I_{l(\jj)}$.		
	\end{lem}

	\begin{proof}
		(i) For $1\le k\le l(\jj)$, use induction and calculate $\beta_{\jj,k}=s_{j_1}s_{j_2}\cdots s_{j_{k-1}}(\alpha_{j_k})=s_{j_2}\cdots s_{j_{k-1}}(\alpha_{j_k})-(s_{j_2}\cdots s_{j_{k-1}}(\alpha_{j_k}),\alpha_{j_1}^\vee)\alpha_{j_1}=s_{j_2}\cdots s_{j_{k-1}}(\alpha_{j_k})+(\beta_{\jj,1},\alpha_{j_1}^\vee)\alpha_{j_1}=\cdots=(\beta_{\jj,k},\beta_{\jj,1}^\vee)\alpha_{j_1}+(\beta_{\jj,k},\beta_{\jj,2}^\vee)\alpha_{j_2}+\cdots+(\beta_{\jj,k},\beta_{\jj,k-1}^\vee)\alpha_{j_{k-1}}+\alpha_{j_k}$.

		(ii) For $1\le k\le l(\jj)$, use induction and calculate $\beta_{\jj,k}=s_{j_1}s_{j_2}\cdots s_{j_{k-1}}(\alpha_{j_k})=s_{j_1}s_{j_2}\cdots s_{j_{k-2}}(\alpha_{j_k}-(\alpha_{j_k},\alpha_{j_{k-1}}^\vee)\alpha_{j_{k-1}})=s_{j_1}s_{j_2}\cdots s_{j_{k-2}}(\alpha_{j_k})-(\alpha_{j_k},\alpha_{j_{k-1}}^\vee)\beta_{\jj,k-1}=\cdots=\alpha_{j_k}-(\alpha_{j_k},\alpha_{j_1}^\vee)\beta_{\jj,1}-(\alpha_{j_k},\alpha_{j_2}^\vee)\beta_{\jj,2}-\cdots-(\alpha_{j_k},\alpha_{j_{k-1}}^\vee)\beta_{\jj,k-1}$.

		(iii) Use induction on $l(\jj)$. When $l(\jj)=1$, there is nothing to prove. Suppose the conclusion is true for $\jj'=(j_1,j_2,\cdots,j_k)$, we want to prove it is ture for reduced expression $\jj=(j_1,j_2,\cdots,j_k,j_{k+1})$. Write
		\[\tp(\jj)=\begin{pmatrix}
			\tp(\jj') & \tilde{\delta}_k\\
			0 & 1
		\end{pmatrix},
		P(\jj)=\begin{pmatrix}
			P(\jj') & \delta_k\\
			0 & 1
		\end{pmatrix},\]
		where
		\[\tilde{\delta}_k=((\alpha_{j_1}^\vee,\alpha_{k+1}),(\alpha_{j_2}^\vee,\alpha_{k+1}),\cdots,(\alpha_{j_k}^\vee,\alpha_{k+1}))\tr,\]
		\[\delta_k=((\beta_{\jj,1}^\vee,\beta_{\jj,k+1}),(\beta_{\jj,2}^\vee,\beta_{\jj,k+1}),\cdots,(\beta_{\jj,k}^\vee,\beta_{\jj,k+1}))\tr.\]
		Then $\tp(\jj)P(\jj)=I_{k+1}$ is equivalent to $P(\jj')\delta_k=-\tilde{\delta}_k$. Denote $\jj''=(j_k,j_{k-1},\cdots,j_1)$, apply (ii) to $\jj''$, which reads
		\begin{align*}
			\alpha_{j_l}&=(\alpha_{j_l},\alpha_{j_k}^\vee)\alpha_{j_k}+(\alpha_{j_l},\alpha_{j_{k-1}}^\vee)s_{j_k}(\alpha_{j_{k-1}})+\cdots\\
			+&(\alpha_{j_l},\alpha_{j_{l+1}}^\vee)s_{j_k}s_{j_{k-1}}\cdots s_{j_{l+2}}(\alpha_{j_{l+1}})+s_{j_k}s_{j_{k-1}}\cdots s_{j_{l+1}}(\alpha_{j_l}),1\le l\le k.
		\end{align*}
		Apply $s_{j_1}s_{j_2}\cdots s_{j_k}$ to both sides of the identity above,
		\[s_{j_1}s_{j_2}\cdots s_{j_k}(\alpha_{j_l})=-\beta_{\jj,l}-(\alpha_{j_l},\alpha_{j_{l+1}}^\vee)\beta_{\jj,l+1}-\cdots-(\alpha_{j_l},\alpha_{j_{k-1}}^\vee)\beta_{\jj,k-1}-(\alpha_{j_l},\alpha_{j_k}^\vee)\beta_{\jj,k},\]
		thus
		\begin{align*}
			-(\alpha_{j_l}^\vee,&\alpha_{j_{k+1}})=(-s_{j_1}s_{j_2}\cdots s_{j_k}(\alpha_{j_l}^\vee),\beta_{\jj,k+1})\\
			&=(\beta_{\jj,l}^\vee,\beta_{\jj,k+1})+(\alpha_{j_l}^\vee,\alpha_{j_{l+1}}^\vee)(\beta_{\jj,l+1},\beta_{\jj,k+1})+\cdots+(\alpha_{j_l}^\vee,\alpha_{j_k}^\vee)(\beta_{\jj,k},\beta_{\jj,k+1})\\
			&=(\beta_{\jj,l}^\vee,\beta_{\jj,k+1})+(\alpha_{j_l}^\vee,\alpha_{j_{l+1}})(\beta_{\jj,l+1}^\vee,\beta_{\jj,k+1})+\cdots+(\alpha_{j_l}^\vee,\alpha_{j_k})(\beta_{\jj,k}^\vee,\beta_{\jj,k+1}),
		\end{align*}
		which implies $P(\jj')\delta_k=-\tilde{\delta}_k$.
	\end{proof}

	Part (ii) of the lemma indicates $\tc(\jj)P(\jj)=C(\jj)$. Part (iii) tells that 
	\[(I_{l(\jj)}+D(\jj)\inv\tb(\jj))(I_{l(\jj)}+D(\jj)\inv B(\jj))=I_{l(\jj)},\]
	thus $B(\jj)=-\tb(\jj)P(\jj)$. So
	\[P(\jj)\tr \ta(\jj) P(\jj)=P(\jj)\tr\tb(\jj)P(\jj)-P(\jj)\tr\tb(\jj)\tr P(\jj)=-P(\jj)\tr B(\jj)+B(\jj)\tr P(\jj)\]
	\[=-(I_{l(\jj)}+B(\jj)\tr D(\jj)\inv)B(\jj)+B(\jj)\tr(I_{l(\jj)}+D(\jj)\inv B(\jj))=-B(\jj)+B(\jj)\tr=-A(\jj).\]
	Denote $Q(\tii)=\diag(P(\tii^+),P(\tii^-),I_n)$, combining these results, one then obtains
	\[Q(\tii)\tr\th(\tii)Q(\tii)=\hh(\tii),\]
	which ends the calculation of $\rr(H(\tii))$.

	\subsection{Pivot Elements and Simple Quotient Modules}

	In general, $\tmii$ as $R_{\tii}$-module admits ``finer" submodule structure than it as $\ttii$-module, i.e. the set of $R_{\tii}$-submodules contains the set of $\ttii$-submodules. However, under some reasonable conditions, the vice versa holds, as we will show in this subsection. This leads to a characterization of simple quotient $R_{\tii}$-modules of $\tmii$ like the $\ttii$ case.

	We need some notations. Given $\uu=\sum_{\aa,\bb\in\Z^m} c_{\aa,\bb}x^{\aa}y^{\bb}\in \lqi,c_{\aa,\bb}\in\C$, denote
	\[\supp(\uu)=\{\aa\oplus\bb\in\Z^{2m}|c_{\aa,\bb}\ne0\},\supp_x(\uu)=\{\aa\in\Z^m|\exists\bb\in\Z^m,c_{\aa,\bb}\ne0\}\]
	and $\supp_x(S_{\tii})=\bigcup_{\uu\in S_{\tii}}\supp_x(\uu)$. For $\aa\in\supp_x(\uu)$, define the \emph{multiplicity of $\aa$ for $\uu$} to be the cardinality of the fiber of the natural projection $\supp(\uu)\twoheadrightarrow\supp_x(\uu)$ at $\aa$. For $1\le k\le m$, let $<_k$ (resp. $\le_k,>_k,\ge_k$) be the binary relation on $\Z^{m}$ or $\Z^{2m}$ given by the the relation $<$ (resp. $\le,>,\ge$) on the $k$-th component. For $I\subset[1,m]$, let $<_I$ (resp. $\le_I,>_I,\ge_I$) be the intersection of all relations $<_k$ (resp. $\le_k,>_k,\ge_k$) for $k\in I$. Use symbol $\odot$ for the component-wise multiplication of column vectors.

	\begin{defn}
		Let $\aa\in(\Z\backslash\{0\})^m,I\subset[1,m],k\in I$, an element $\uu\in \lqi$ is called an \emph{$\aa$-pivot element of type $(I,k)$} if $\supp_x(\uu)$ has an element $\ccc$ of multiplicity 1 for $\uu$ such that $\aa\odot\ccc\le_I0,\aa\odot\ccc<_k0$ and for any other $\ccc'\in\supp(\uu),\aa\odot(\ccc'-\ccc)\ge_I0,\aa\odot(\ccc'-\ccc)>_k0$.
	\end{defn}

	\begin{defn}\label{defn-enoughpv}
		We say there are \emph{enough pivot elements} in $R_{\tii}$ if for some permutation $n_1,n_2,\cdots,n_m$ of $[1,m]$ there exists $\aa_k$-pivot element $\uu_k$ of type $([1,m]\backslash\{n_1,n_2,\cdots,n_{k-1}\},n_k)$ in $R_{\tii}$ for some $\aa_k\in\supp_x(S_{\tii})\cap(\Z\backslash\{0\})^m, 1\le k\le m$.
	\end{defn}

	Here is one sufficient condition for $R_{\tii}$ having enough pivot elements. In the next subsection we will show more examples when $G=SL_3(\C)$.

	\begin{prop}\label{prop-expivot}
		If $\supp(w_1)\cap\supp(w_2)=\varnothing$, then $R_{\tii}$ has enough pivot elements.
	\end{prop}

	\begin{proof}
		Let $\mu\in P$. $\pp_{\tii,\mu,\mu}$ consisits of weight strings of the form $(\mu_0,\mu_1,\cdots,\mu_m),\mu_0=\mu_m=\mu,\mu_{k-1}-\mu_k=j_k\sgn(\ti_k)\alpha_{|\ti_k|},j_k\ge0,1\le k\le m$. Thus
		\[\sum_{\ti_k>0}j_k\alpha_{|\ti_k|}=\sum_{\ti_k<0}j_k\alpha_{|\ti_k|}.\]
		Since $\supp(w_1)\cap\supp(w_2)=\varnothing$, we have $j_k=0,1\le k\le m$. Therefore $\pp_{\tii,\mu,\mu}=\{\muu_0=(\mu,\mu,\cdots,\mu)\}$. Similarly $\pp_{\tii,-\mu,-\mu}=\{-\muu_0=(-\mu,-\mu,\cdots,-\mu)\}$. So by Proposition \ref{prop-pii},
		\[I(\muu_0)=x^{\aa(\muu_0)}\in S_{\tii},I(-\muu_0)=x^{-\aa(\muu_0)}\in S_{\tii},\]
		where $\aa(\muu_0)=((\mu,\alpha_{|\ti_1|}^\vee),(\mu,\alpha_{|\ti_2|}^\vee),\cdots,(\mu,\alpha_{|\ti_m|}^\vee))$. Choose $\mu\in P$ such that $\aa(\muu_0)\in(\Z\backslash\{0\})^m$, then $I(-\muu_0)$ is an $I(\muu_0)$-pivot element of arbitrary type, so the conclusion follows.
	\end{proof}

	\begin{cor}\label{cor-expivot}
		If $w_1=e$ or $w_2=e$, then $R_{\tii}$ has enough pivot elements.
	\end{cor}

	Recall the notion of $x(\aa)\in\ttiit,\aa\in\Z^m$ in Remark \ref{rem-xaa}. Let
	\[\epss_i=(0,\cdots,0,\underset{i\text{-th}}{1},0,\cdots,0),1\le i\le m\]
	be the standard basis of $\Z^m$. We will illustrate the consequense of $R_{\tii}$ having enough pivot elements.

	\begin{prop}\label{prop-ritiequiv}
		If $R_{\tii}$ has enough pivot elements, then for any weight vector $\tee_{\mm}\in\tmii(\mm),\mm\in\Z^{s(\tii)}$, $R_{\tii}\tee_{\mm}=\ttii\tee_{\mm}$.
	\end{prop}

	\begin{proof}
		\sloppy Let $n_1,n_2,\cdots,n_m\in[1,m],\aa_k=(a_{k1},a_{k2},\cdots,a_{km})\tr,\uu_k\in\Z^m,1\le k\le m$ be as in Definition \ref{defn-enoughpv}. Denote $\cM:=\{\aa\in\Z^m|x(\aa)\tee_{\mm}\in R_{\tii}\tee_{\mm}\}$, it suffices to show $\cM=\Z^m$.

		\emph{Step1.} We prove the following fact: for $1\le k\le m$, there exists $\ccc_k'\in\Z^m$ such that for any $\ccc\in\cM$,
		\[\ccc+\ccc_k'+\sum_{l=1}^m\N\sgn(a_{kl})\epss_l\subset\cM.\]
		Write $x(\epss_j)=\ds\inv\rrr_j$, where $\ds,\rrr_j\in S_{\tii},1\le j\le m$. Then $\ds\in x(\taa)\ttizt,\rrr_j\in x(\taa+\epss_j)\ttizt$ for some $\taa\in\Z^m$. Let $\taa_k=\taa+t_k\aa_k$, $t_k\in\N$ large enough such that $\aa_k\odot\taa_k>_{[1,m]}0$. We have $\taa_k,\taa_k+\epss_i\in\supp_x(S_{\tii})$, thus $\N\taa_k+\sum_{i=1}^m\N(\taa_k+\epss_i)\subset\supp_x(S_{\tii})$. Let $\taa_k=(\tla_{k1},\tla_{k2},\cdots,\tla_{km})\tr$, note that
		\[|\tla_{ki}|\taa_k+j\epss_i=(|\tla_{ki}|-j)\taa_k+j(\taa_k+\epss_i)\in\supp_x(S_{\tii}),0\le j\le|\tla_{ki}|,1\le i\le m,\]
		i.e. $|\tla_{ki}|\taa_k+[0,|\tla_{ki}|]\epss_i\subset\supp_x(S_{\tii}),1\le i\le m$, therefore
		\begin{align*}
			\supp_x(S_{\tii})&\supset\N\taa_k+\sum_{i=1}^m(|\tla_{ki}|\taa_k+[0,|\tla_{ki}|]\epss_i)\\
			&\supset(\sum_{i=1}^m|\tla_{ki}|)\taa_k+\N\taa_k+[0,|\tla_{k1}|]\times[0,|\tla_{k2}|]\times\cdots\times[0,|\tla_{km}|]\\
			&=(\sum_{i=1}^m|\tla_{ki}|)\taa_k+\sgn(a_{k1})\N\times\sgn(a_{k2})\N\times\cdots\times\sgn(a_{km})\N.
		\end{align*}
		Consider $S_{\tii}$'s action on $x(\ccc)\tee_{\mm}$, we see $\ccc+\supp_x(S_{\tii})\subset\cM$. So the fact follows from letting $\ccc_k'=(\sum_{i=1}^m|\tla_{ki}|)\taa_k$.
		
		\emph{Step 2.} Denote $I_k=\{n_k,n_{k+1},\cdots,n_m\},1\le k\le m$ and $I_{m+1}=\varnothing$. Assume $\supp_x(\uu_k)=\{\bb_k,\bb_k+\bb_{k1},\bb_k+\bb_{k2},\cdots,\bb_k+\bb_{kp_k}\}$ where $\bb_k$ has multiplicity 1 for $\uu_k$, $\aa_k\odot\bb_k\le_{I_k}0,\aa_k\odot\bb_k<_{n_k}0,\aa_k\odot\bb_{ki}\ge_{I_k}0,\aa_k\odot\bb_{ki}>_{n_k}0,1\le i\le p_k,p_k\ge0$. For $\aa\in\Z^m$, let $\aa'=\aa-\sgn(a_{kn_k})\epss_{n_k}$. Consider $\uu_k$'s action on $x(\aa'-\bb_k)\tee_{\mm}$, we see if $\aa'-\bb_k\in\cM$ and $\aa'+\bb_{ki}\in\cM,1\le i\le p_k$, then $\aa'\in\cM$. Since
		\[\aa'-\bb_k,\aa'+\bb_{ki}\in\aa+\sum_{l\in I_k}\N\sgn(a_{kl})\epss_l+\sum_{l\notin I_k}\Z\epss_l,\]
		we conclude that if
		\[\aa+\sum_{l\in I_k}\N\sgn(a_{kl})\epss_l+\sum_{l\notin I_k}\Z\epss_l\subset\cM,\]
		then $\aa'=\aa-\sgn(a_{kn_k})\epss_{n_k}\in\cM$ and thus $\aa-\sgn(a_{kn_k})\epss_{n_k}+\sum_{l\in I_k}\N\sgn(a_{kl})\epss_l+\sum_{l\notin I_k}\Z\epss_l\subset\cM$, which by induction implies
		\[\aa-\N\sgn(a_{kn_k})\epss_{n_k}+\sum_{l\in I_k}\N\sgn(a_{kl})\epss_l+\sum_{l\notin I_k}\Z\epss_l\subset\cM,\]
		i.e.
		\[\aa+\sum_{l\in I_{k+1}}\N\sgn(a_{kl})\epss_l+\sum_{l\notin I_{k+1}}\Z\epss_l\subset\cM.\]

		\emph{Step 3.} We use induction on $k$ to prove the following claim: for $1\le k\le m$, there exists $\ccc_k\in\Z^m$ such that
		\[\ccc_k+\sum_{l\in I_{k+1}}\N\sgn(a_{kl})\epss_l+\sum_{l\notin I_{k+1}}\Z\epss_l\subset\cM.\]
		In \emph{Step 1}, choose $k=1$ and $\ccc=0$, then by \emph{Step 2} we obtain the claim for $k=1$. Suppose the claim is true for some $k\le m-1$, then by \emph{Step 1} for $k+1$ we see
		\begin{align*}
			\ccc_k+\ccc_{k+1}'+&\sum_{l\in I_{k+1}}\N\sgn(a_{(k+1)l})\epss_l+\sum_{l\notin I_{k+1}}\Z\epss_l\\
			\subset&\Big(\ccc_k+\sum_{l\in I_{k+1}}\N\sgn(a_{kl})\epss_l+\sum_{l\notin I_{k+1}}\Z\epss_l\Big)+\Big(\ccc_{k+1}'+\sum_{l=1}^m\N\sgn(a_{(k+1)l})\epss_l\Big)\subset\cM.
		\end{align*}
		The claim for $k+1$ then follows from \emph{Step 2} if we choose $\ccc_{k+1}=\ccc_k+\ccc_{k+1}'$. Finally, the claim for $k=m$ leads to the proposition.
	\end{proof}

	\begin{cor}\label{cor-ritiequiv}
		If $R_{\tii}$ has enough pivot elements, then for any $\tee\in\tmii,R_{\tii}\tee=\ttii\tee$.
	\end{cor}

	\begin{proof}
		Let $\tee=\sum_{\mm\in\Z^{s(\tii)}} \tee_{\mm}$, where $\tee_{\mm}\in\tmii(\mm)$ is all but finitely 0. A Vandermonde determinant argument shows there exists $\tt_{\mm}\in\ttiz$ such that $\tee_{\mm}=\tt_{\mm}\tee,\forall\mm\in\Z^{s(\tii)}$. Write $\tt_{\mm}=\ds_{\mm}\inv\rrr_{\mm},\ds_{\mm}\in S_{\tii},\rrr_{\mm}\in R_{\tii}$, then $\ds_{\mm}\tee_{\mm}=\rrr_{\mm}\tee$, i.e. $\ds_{\mm}\tee_{\mm}\in R_{\tii}\tee,\forall\mm\in\Z^{s(\tii)}$. The proposition above shows $\tee_{\mm}\in\ttii\ds_{\mm}\tee_{\mm}=R_{\tii}\ds_{\mm}\tee_{\mm}\subset R_{\tii}\tee$. Therefore by the proposition again,
		\[\ttii\tee=\ttii\sum_{\mm\in\Z^{s(\tii)}}\tee_{\mm}\subset\sum_{\mm\in\Z^{s(\tii)}}\ttii\tee_{\mm}=\sum_{\mm\in\Z^{s(\tii)}} R_{\tii}\tee_{\mm}\subset\sum_{\mm\in\Z^{s(\tii)}} R_{\tii}\tee=R_{\tii}\tee,\]
		thus $R_{\tii}\tee=\ttii\tee$.
	\end{proof}

	In view of Theorem \ref{thm-simquot} and the corollary above, we have proved the following theorem.

	\begin{thm}\label{thm-main}
		If $R_{\tii}$ has enough pivot elements, then $\tmii$ as $R_{\tii}$-module is indecomposable, and its simple quotient modules are in bijection with simple $\cc_{\tii}'$-modules.
	\end{thm}

	Notice Proposition \ref{prop-prim} and (\ref{eq-tirrti}), we have the following corollary.

	\begin{cor}\label{cor-bundle}
		If $R_{\tii}$ has enough pivot elements, then the $\tirr\ttii$ defined in (\ref{eq-tirrti}) becomes a bundle of $(\ww)$ type simple $\cqg$-modules onto $\prim\cqg_{\ww}$ by taking annihilators, with fiber
		\[\irr L_{q^{m_{\tii,1}}}(2)\times\irr L_{q^{m_{\tii,2}}}(2)\times\cdots\times\irr L_{q^{m_{\tii,k(\tii)}}}(2).\]
	\end{cor}

	As an application we deduce the equivalent condition for $\tmii$ being simple.

	\begin{thm}\label{thm-irrtenmod}
		The $\cqg$-module $\tmii$ is simple if and only if $s(\tii)=m$, i.e. $|\ti_1|,|\ti_2|,\cdots,|\ti_m|$ are distinct.
	\end{thm}

	\begin{proof}
		If the $R_{\tii}$-module $\tmii$ is simple, so is the $\ttii$-module $\tmii$. Corollary \ref{cor-ccdecomp} then tells that $\cc_{\tii}'$ must be trivial, so the dimension  $m-s(\tii)$ of $\cc_{\tii}'$ as quantum torus is 0. On the contrary, if $|\ti_1|,|\ti_2|,\cdots,|\ti_m|$ are distinct, then $\supp(w_1)\cap\supp(w_2)=\varnothing$, so $R_{\tii}$ has enough pivot elements by Proposition \ref{prop-expivot}. Apply Theorem \ref{thm-main}, since $\cc_{\tii}'=\C$ is trivial, $\tmii$ is simple.
	\end{proof}

	\subsection{Wiring Diagrams and $G=SL_3(\C)$ Case}

	In this subsection we apply the theory of previous subsections to the type A case. Let $G=SL_{n+1}$. Take simple roots as $\alpha_{i}=\varepsilon_{i}-\varepsilon_{i+1}$, where $\varepsilon_1,\varepsilon_2,\cdots,\varepsilon_{n+1}$ form a orthonormal basis in some Euclidean space. The Weyl group $W\simeq S_{n+1}$ is generated by simple reflections $s_i=(i,i+1),1\le i\le n$. Recall in Remark \ref{rem-quanmat}, $\C_q[SL_{n+1}]$ is generated by $x_{ij},1\le i,j\le n+1$ subject to relations (\ref{eq-quanmat}). For subsets $A,B\subset[1,n+1],|A|=|B|$, define the \emph{quantum minor} $\Delta_{q,A,B}$ to be the quantum determinant of the submatrix of $(x_{ij})$ with row set $A$ and column set $B$. Denote $A_{ij}=\Delta_{q,[1,n+1]\backslash\{i\},[1,n+1]\backslash\{j\}}$. The Hopf algebra structure of $\C_q[SL_{n+1}]$ is then given by
	\begin{equation}\label{eq-hopfquanmat}
		\Delta(x_{ij})=\sum_{k=1}^{n+1}x_{ik}\otimes x_{kj},\varepsilon(x_{ij})=\delta_{ij},S(x_{ij})=(-q)^{j-i}A_{ji}.
	\end{equation}

	Our idea of the \emph{wiring diagram} is to study the actions of quantum minors on $\tmii$ in a combinatorial way, which is slightly different from that of Fomin and Zelevinsky in \cite{FZ}*{Section 4.1}. For $\tii=\tiim\in\I_{\ww}$, construct the wiring diagram $\Gamma(\tii)$ as the concatenation of the ``elementary'' wiring diagrams $\Gamma_{\ti_1},\Gamma_{\ti_2},\cdots,\Gamma_{\ti_m}$ from left to right, where $\Gamma_{i}$ is constructed by numbering $n+1$ horizontal edges with $1,2\cdots,n+1$ from bottom to top and drawing a diagonal edge connecting level $|i|$ and $|i|+1$ depending on $\sgn(i)$, c.f. Figure \ref{fig-elmwd}.

	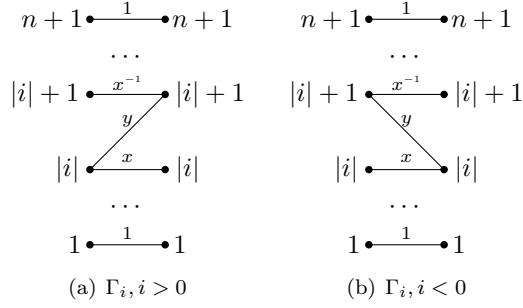
\begin{figure}[H]
		\centering
		\subfigure[$\Gamma_i,i>0$]{
			\begin{tikzpicture}
				\draw (0,0)--(1,0);
				\draw (0,3)--(1,3);
				\node at (0,0) [circle,fill,inner sep=1pt] {};
				\node at (1,0) [circle,fill,inner sep=1pt] {};
				\node at (0,3) [circle,fill,inner sep=1pt] {};
				\node at (1,3) [circle,fill,inner sep=1pt] {};
				\node at (0,2) [circle,fill,inner sep=1pt] {};
				\node at (1,1) [circle,fill,inner sep=1pt] {};
				\node at (-0.2,0) {$1$};
				\node at (-0.3,1) {$|i|$};
				\node at (-0.6,2) {$|i|+1$};
				\node at (-0.5,3) {$n+1$};
				\node at (1.2,0) {$1$};
				\node at (1.3,1) {$|i|$};
				\node at (1.6,2) {$|i|+1$};
				\node at (1.5,3) {$n+1$};
				\node at (0.5,0.5) {$\cdots$};
				\node at (0.5,2.5) {$\cdots$};
				\node at (0.5,0.15) {$\scriptstyle 1$};
				\node at (0.5,3.15) {$\scriptstyle 1$};
				\boxl{1}{2}
			\end{tikzpicture}
		}
		\subfigure[$\Gamma_i,i<0$]{
			\begin{tikzpicture}
				\draw (0,0)--(1,0);
				\draw (0,3)--(1,3);
				\node at (0,0) [circle,fill,inner sep=1pt] {};
				\node at (1,0) [circle,fill,inner sep=1pt] {};
				\node at (0,3) [circle,fill,inner sep=1pt] {};
				\node at (1,3) [circle,fill,inner sep=1pt] {};
				\node at (0,1) [circle,fill,inner sep=1pt] {};
				\node at (1,2) [circle,fill,inner sep=1pt] {};
				\node at (-0.2,0) {$1$};
				\node at (-0.3,1) {$|i|$};
				\node at (-0.6,2) {$|i|+1$};
				\node at (-0.5,3) {$n+1$};
				\node at (1.2,0) {$1$};
				\node at (1.3,1) {$|i|$};
				\node at (1.6,2) {$|i|+1$};
				\node at (1.5,3) {$n+1$};
				\node at (0.5,0.5) {$\cdots$};
				\node at (0.5,2.5) {$\cdots$};
				\node at (0.5,0.15) {$\scriptstyle 1$};
				\node at (0.5,3.15) {$\scriptstyle 1$};
				\boxl{1}{-2}
			\end{tikzpicture}
		}
		\caption{\label{fig-elmwd}Elementary wiring diagrams}
	\end{figure}

	Let $x^\pm,y^\pm$ be the generators of $\lqt{}$ such that $xy=qyx$. In each $\Gamma_i$, we associate weights $x,x\inv,y,1$ to the edge of level $|i|$, the edge of level $|i|+1$, the edge connecting level $|i|$ and $|i|+1$, and all other edges respectively, c.f. Figure \ref{fig-elmwd}. All edges are given the orientation from left to right. Figure \ref{fig-wd} shows an example of $\Gamma(\tii)$ when $n=3$, with all edges labeled by weights.

	\begin{figure}[H]
		\centering
		\begin{tikzpicture}
			\boxdiagraml{4}{{-2,1,-3,3,2,-1,-2,1,-1}}
		\end{tikzpicture}
		\caption{\label{fig-wd}$\Gamma(\tii)$ for $\tii=(-2,1,-3,3,2,-1,-2,1,-1)$}
	\end{figure}

	Denote the set of paths in $\Gamma(\tii)$ from level $i$ in the left side to level $j$ in the right side by $\Gamma(\tii)_{ij}$. For a path $p$ in $\Gamma(\tii)$, define its weight $I(p)\in\lqi$ to be the tensor product of weights of all edges in the path from left to right. Notice for $1\le k\le n$,
	\begin{equation}\label{eq-quanmatres}
		\iota_k^\pm\circ\tpi_k^\pm(x_{ij})=\begin{cases}
			1, & i=j\notin\{k,k+1\},\\
			x, & i=j=k,\\
			x\inv, & i=j=k+1,\\
			y, & i=(2k+1\mp1)/2,j=(2k+1\pm1)/2,\\
			0, & \text{otherwise}.
		\end{cases}
	\end{equation}
	Therefore one can easily check the following formula, in light of (\ref{eq-hopfquanmat}),
	\begin{equation}\label{eq-tpixij}
		\tpi_{\tii}(x_{ij})=\sum_{p\in\Gamma(\tii)_{ij}}I(p).
	\end{equation}

	We generalize (\ref{eq-tpixij}) to the quantum minor case, i.e. a quantum version of Lindstr\"{o}m's lemma. Let $A,B\subset[1,n+1],|A|=|B|=k$, a family of $k$ paths $\ppp=\{p_1,p_2,\cdots,p_k\}$ is called \emph{$(A,B)$ type}, if $p_1,p_2,\cdots,p_k$ are disjoint paths from levels in $A$ in the left side to levels in $B$ in the right side, c.f. Figure \ref{fig-fampath}. Define weights of such $\ppp$'s to be $I(\ppp)=I(p_1)I(p_2)\cdots I(p_k)$, which is independent of the order of $p_i$'s. Denote the set of families of paths of type $(A,B)$ in $\Gamma(\tii)$ by $\Gamma(\tii)_{A,B}$. We are ready to state the following proposition.

	\begin{figure}[H]
		\centering
		\begin{tikzpicture}
			\boxdiagram{4}{{-2,1,-3,3,2,-1,-2,1,-1}}
			\drawpath{{3,2,2,2,2,3,3,3,3,3}}
			\drawpath{{1,1,1,1,1,1,1,1,2,1}}
		\end{tikzpicture}
		\caption{\label{fig-fampath}A type $(\{1,3\},\{1,3\})$ family of paths in $\Gamma(\tii)$}
	\end{figure}
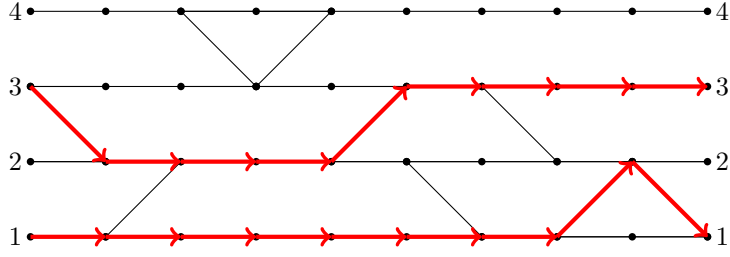

	\begin{prop}[Quantum Lindstr\"{o}m's lemma]\label{prop-quanlind}
		Keep notations above, then
		\[\tpi_{\tii}(\Delta_{q,A,B})=\sum_{\ppp\in\Gamma(\tii)_{A,B}}I(\ppp).\]
	\end{prop}

	\begin{proof}
		Assume elements of $A$ are $l_1<l_2<\cdots<l_k$ and elements of $B$ are $l_1'<l_2'<\cdots<l_k'$. Denote $y_{ij}=x_{l_il_j'},1\le i,j\le k$. Then by definition of $\Delta_{q,A,B}$,
		\begin{equation}\label{eq-tpidelta}
			\tpi_{\tii}(\Delta_{q,A,B})=\sum_{\tau\in S_k}(-q)^{l(\tau)}\tpi_{\tii}(y_{1\tau(1)})\tpi_{\tii}(y_{2\tau(2)})\cdots\tpi_{\tii}(y_{k\tau(k)}).
		\end{equation}
		Use induction on $m=l(\tii)$. If $m=1$, then in (\ref{eq-tpidelta}), all summand terms corresponding to $\tau\ne e$ are zero. The conclusion then follows from (\ref{eq-quanmatres}). Suppose the proposition is true for $\tii'=(\ti_1,\ti_2,\cdots,\ti_{m-1})$, we prove it is true for $\tii$. Without loss of generality we may assume $\ti_m>0$. By (\ref{eq-tpixij}) we have (c.f. Figure \ref{fig-tpixij})
		\begin{equation}\label{eq-tpiind}
			\tpi_{\tii}(x_{ij})=\begin{cases}
				\tpi_{\tii'}(x_{ij})\otimes 1, & j\ne i_m,i_m+1,\\
				\tpi_{\tii'}(x_{ij})\otimes x, & j=i_m,\\
				\tpi_{\tii'}(x_{ij})\otimes x\inv+\tpi_{\tii'}(x_{i(j-1)})\otimes y, & j=i_m+1.
			\end{cases}
		\end{equation}

		\begin{figure}[h]
			\centering
			\begin{tikzpicture}
				\draw[thick,dashed] (0,0)--(0,4) (3,0)--(3,4) (4,0)--(4,4);
				\node at (-0.2,1) {$\scriptstyle i$};
				\node at (4.3,2) {$\scriptstyle i_m$};
				\node at (4.45,3) {$\scriptstyle i_m+1$};
				\node at (4.85,1) {$\scriptstyle j\ne i_m,i_m+1$};
				\node at (3.5,1.2) {$\scriptstyle1$};
				\node at (1.2,3.1) {$\scriptstyle\tpi_{\tii'}(x_{i,i_m+1})$};
				\node at (1.8,1.5) {$\scriptstyle\tpi_{\tii'}(x_{i,i_m})$};
				\node at (1.6,0.3) {$\scriptstyle\tpi_{\tii'}(x_{ij})$};
				\node (i) at (0,1) [circle,fill,inner sep=1pt] {};
				\node (k) at (3,2) {};
				\node (kp1) at (3,3) [circle,fill,inner sep=1pt] {};
				\node at (4,2) [circle,fill,inner sep=1pt] {};
				\node (j) at (3,1) {};
				\draw (i) to[bend left] (k);
				\draw (i) to[bend left] (kp1);
				\draw (i) to[bend right] (j);
				\draw (3,1)--(4,1);
				\node at (3,1) [circle,fill,inner sep=1pt] {};
				\node at (4,1) [circle,fill,inner sep=1pt] {};
				\boxl{4}{3};
				\draw [thick,decorate,decoration={brace,amplitude=10pt,mirror},xshift=0.4pt,yshift=-0.4pt](0,0) -- (4,0) node[black,midway,yshift=-0.6cm] {$\scriptstyle\Gamma(\tii)$};
				\draw [thick,decorate,decoration={brace,amplitude=10pt},xshift=0.4pt,yshift=-0.4pt](0,4) -- (3,4) node[black,midway,yshift=0.6cm] {$\scriptstyle\Gamma(\tii')$};
			\end{tikzpicture}
			\caption{Calculaion of $\tpi_{\tii}(x_{ij})$\label{fig-tpixij}}
		\end{figure}

		For any $\ppp=\{p_1,p_2,\cdots,p_k\}\in\Gamma(\tii)_{A,B}$, denote $\ppp'=\{p_1',p_2',\cdots,p_k'\}$, where $p_i'$ is the intersection of $p_i$ with $\Gamma(\tii')$. We will check the induction step case by case (c.f. Figure \ref{fig-tpixij}).
		
		(i) If $\ti_m,i_m+1\notin B$, then with (\ref{eq-tpiind}), the right side of (\ref{eq-tpidelta}) becomes $\tpi_{\tii'}(\Delta_{q,A,B})\otimes 1$, which results in $\sum_{\tpp\in\Gamma(\tii')_{A,B}}I(\tpp)\otimes 1$ by induction hypothesis. On the other hand, $\Gamma(\tii)_{A,B}\simeq\Gamma(\tii')_{A,B}$ via $\ppp\mapsto\ppp',I(\ppp)=I(\ppp')\otimes 1,\ppp\in\Gamma(\tii)_{A,B}$. So the proposition holds for $\tii$ in this case.

		(ii) If $\ti_m\in B,\ti_m+1\notin B$, then with (\ref{eq-tpiind}), the right side of (\ref{eq-tpidelta}) becomes $\tpi_{\tii'}(\Delta_{1,A,B})\otimes x$, which results in $\sum_{\tpp\in\Gamma(\tii')_{A,B}}I(\tpp)\otimes x$ by induction hypothesis. On the other hand, $\Gamma(\tii)_{A,B}\simeq\Gamma(\tii')_{A,B}$ via $\ppp\mapsto\ppp',I(\ppp)=I(\ppp')\otimes x,\ppp\in\Gamma(\tii)_{A,B}$. So the proposition holds for $\tii$ in this case.

		(iii) If $\ti_m\notin B,\ti_m+1=l_h'\in B$, let $B'=(B\backslash\{\ti_m+1\})\cup\{\ti_m\}$. Then with (\ref{eq-tpiind}), the right side of (\ref{eq-tpidelta}) becomes $\tpi_{\tii'}(\Delta_{q,A,B})\otimes x\inv+\tpi_{\tii'}(\Delta_{q,A,B'})\otimes y$, which results in $\sum_{\tpp\in\Gamma(\tii')_{A,B}}I(\tpp)\otimes x\inv+\sum_{\tpp\in\Gamma(\tii')_{A,B'}}I(\tpp)\otimes y$ by induction hypothesis. On the other hand, $\Gamma(\tii)_{A,B}\simeq\Gamma(\tii')_{A,B}\sqcup\Gamma(\tii')_{A,B'}$ via $\ppp\mapsto\ppp'$, and for $\ppp\in\Gamma(\tii)_{A,B},I(\ppp)=I(\ppp')\otimes x\inv$ or $I(\ppp)=I(\ppp')\otimes y$ depending on $\ppp'\in\Gamma(\tii')_{A,B}$ or $\ppp'\in\Gamma(\tii')_{A,B'}$. So the proposition holds for $\tii$ in this case.

		(iv) If $\ti_m=l_h'\in B,\ti_m+1=l_{h+1}'\in B$, let $y_{ij}'=y_{ij}=x_{l_il_j'},j\ne h+1,y_{i(h+1)}'=y_{ih}'=x_{l_ih}$. Then with (\ref{eq-tpiind}), the right side of (\ref{eq-tpidelta}) becomes $\tpi_{\tii'}(\Delta_{q,A,B})\otimes 1+I_0$, where
		\[I_0=\sum_{\tau\in S_k}(-q)^{l(\tau)}\tpi_{\tii'}(y_{1\tau(1)}'y_{2\tau(2)}'\cdots y_{k\tau(k)}')\otimes I_\tau,I_\tau=\begin{cases}
			xy, & \tau\inv(h)<\tau\inv(h+1), \\
			yx, & \tau\inv(h)>\tau\inv(h+1).
		\end{cases}\]
		Denote $S_k(h)=\{\tau\in S_k|\tau\inv(h)<\tau\inv(h+1)\}$. Then for $\tau\in S_k(h),l(s_h\tau)=l(\tau)+1$ and $S_k=S_k(h)\sqcup s_hS_k(h)$. So we can rewrite $I_0$ as
		\[I_0=\sum_{\tau\in S_k(h)}\tpi_{\tii'}(y_{1\tau(1)}'y_{2\tau(2)}'\cdots y_{k\tau(k)}')\otimes((-q)^{l(\tau)}xy+(-q)^{l(s_h\tau)}yx)=0.\]
		By induction hypothesis, $\tpi_{\tii'}(\Delta_{q,A,B})\otimes 1$ becomes $\sum_{\tpp\in\Gamma(\tii')_{A,B}}I(\tpp)\otimes 1$. On the other hand, $\Gamma(\tii)_{A,B}\simeq\Gamma(\tii')_{A,B}$ via $\ppp\mapsto\ppp',I(\ppp)=I(\ppp')\otimes 1,\ppp\in\Gamma(\tii)_{A,B}$. So the proposition holds for $\tii$ in this case.
	\end{proof}

	Now let $G=SL_3(\C)$, we want to construct some simple modules for each primitive ideal using the bundle $\tirr\ttii$, c.f. Corollary \ref{cor-bundle}. This involves constructions of enough pivot elements in $R_{\tii}$ for some $\tii\in\I_{\ww},\forall(\ww)\in\wtw\simeq S_3\times S_3$. Write $S_3=\{(1),(12),(23),(123),(132),(13)\}$, with longest element $w_0=(13)$. Our main tool of calculation is Quantum Lindstr\"{o}m's lemma (Proposition \ref{prop-quanlind}). First we rule out the cases described by Proposition \ref{prop-expivot}, i.e. when $\supp(w_1)\cap\supp(w_2)=\varnothing$. Besides, we notice that a horizontal or vertical flip of wiring diagrams does not affect the exsitence of enough pivot elements. So the constructions for $(\ww)$ case work equivalently for $(w_2\inv,w_1\inv)$ and $(w_0\inv w_2w_0,w_0\inv w_1w_0)$ case, after suitable adjustments according to the flip. We list constructions for all the remaining non-equivalent cases in Table \ref{tab-constructions}. Here we adopt the notations $n_k,\aa_k,\uu_k$ as in Definition \ref{defn-enoughpv}, and denote $\xx_k,\xx_k'\in \C_q[SL_3]$ such that $\supp_x(\tpi_{\tii}(\xx_k))=\{\aa_k\},\tpi_{\tii}(\xx_k')=\uu_k$.

	\begin{ex}
		We check the construction of enough pivot elements for $(\ww)=((123),(13)),\tii=(1,2,1,-1,-2)$ in Table \ref{tab-constructions}. Draw the wiring diagram $\Gamma(\tii)$ as in Figure \ref{fig-exwd}. Using Quantum Lindstr\"{o}m's lemma and calculate
		\begin{align*}
			\tpi_{\tii}(\xx_1)&=\tpi_{\tii}(\xx_2)=\tpi_{\tii}(x_{21}^2x_{33}\Delta_{q,23,23})=x^{-3}\otimes x\otimes x^{-3}\otimes y^2x\inv\otimes x\inv,\\
			\tpi_{\tii}(\xx_3)&=\tpi_{\tii}(\xx_4)=\tpi_{\tii}(\xx_5)=\tpi_{\tii}(x_{33}\Delta_{q,23,23})=x\inv\otimes x\inv\otimes x\inv\otimes x\inv\otimes x\inv,\\
			\tpi_{\tii}(\xx_1')&=\tpi_{\tii}(\xx_2')=\tpi_{\tii}(\Delta_{q,13,12})\\
			&=y\otimes 1\otimes x\inv\otimes y\inv\otimes y+x\otimes x\inv\otimes y\otimes y\otimes y+x\otimes x\inv\otimes x\otimes x\otimes y,\\
			\tpi_{\tii}(\xx_3')&=\tpi_{\tii}(\xx_4')=\tpi_{\tii}(\Delta_{q,12,12})\\
			&=1\otimes y\otimes y\otimes y\otimes y+1\otimes y\otimes x\otimes x\otimes y+1\otimes x\otimes 1\otimes 1\otimes x,\\
			\tpi_{\tii}(\xx_5)&=\tpi_{\tii}(x_{12})\\
			&=y\otimes y\otimes 1\otimes 1\otimes y+y\otimes x\otimes x\inv\otimes x\inv\otimes x+x\otimes 1\otimes y\otimes x\inv\otimes x.
		\end{align*}
		Thus
		\begin{align*}
			&\aa_1=\aa_2=(-3,1,-3,-1,-1)\tr,\\
			&\aa_3=\aa_4=\aa_5=(-1,-1,-1,-1,-1)\tr,\\
			&\supp_x(\uu_1)=\supp_x(\uu_2)=\{(0,0,-1,-1,0)\tr,(1,-1,0,0,0)\tr,(1,-1,1,1,0)\tr\},\\
			&\supp_x(\uu_3)=\supp_x(\uu_4)=\{(0,0,0,0,0)\tr,(0,0,1,1,0)\tr,(0,1,0,0,1)\tr\},\\
			&\supp_x(\uu_5)=\{(0,0,0,0,0)\tr,(0,1,-1,-1,1)\tr,(1,0,0,-1,1)\tr\}.
		\end{align*}
		So we see $\uu_1$ is $\aa_1$-pivot element of type $([1,5],3)$, $\uu_2$ is $\aa_2$-pivot element of type $(\{1,2,4,5\},4)$, $\uu_3$ is $\aa_3$-pivot element of type $(\{1,2,5\},2)$, $\uu_4$ is $\aa_4$-pivot element of type $(\{1,5\},5)$, $\uu_5$ is $\aa_5$-pivot element of type $(\{1\},5)$. Therefore $(n_1,n_2,n_3,n_4,n_5)=(3,4,2,5,1)$.
	\end{ex}

	\begin{figure}[h]
		\centering
		\begin{tikzpicture}
			\boxdiagraml{3}{{1,2,1,-1,-2}}
		\end{tikzpicture}
		\caption{\label{fig-exwd}$\Gamma(\tii)$ when $\tii=(1,2,1,-1,-2)$}
	\end{figure}

	\begin{table}[h]
		\addtolength{\tabcolsep}{-2pt}
		\fontsize{9pt}{12pt}\selectfont
		\centering
		\caption{\label{tab-constructions}Constructions of enough pivot elements for $G=SL_3(\C)$}
		\vspace{0.2\baselineskip}
		\begin{tabular}{c|c|c|c|c}
			\hline
			$\ww$ & $\tii\in\I_{\ww}$ & $(n_k)$ & $\xx_k$ & $\xx_k'$ \\
			\hline
			(12),(12) & (-1,1) & (1,2) & $\xx_1=\xx_2=x_{11}$ & $\xx_1'=\xx_2'=x_{22}$\\
			\hline
			(12),(123) & (-1,1,2) & (1,2,3) & $\begin{array}{c}\xx_1=\xx_2=\xx_3\\=x_{11}x_{33}\end{array}$ & $\begin{array}{c}\xx_1'=\xx_2'=x_{22}\\\xx_3'=x_{12}\end{array}$\\
			\hline
			(12),(132) & (-1,2,1) & (1,2,3) & $\begin{array}{c}\xx_1=\xx_2=\xx_3\\=x_{11}x_{33}\end{array}$ & $\xx_1'=\xx_2'=\xx_3'=x_{22}$\\
			\hline
			(12),(13) & (-1,1,2,1) & (1,2,3,4) & $\begin{array}{c}\xx_1=\xx_2=\xx_3=\xx_4\\=x_{11}\Delta_{q,12,12}\end{array}$ & $\begin{array}{c}\xx_1'=\xx_2'=x_{23}\\\xx_3'=x_{33}\\\xx_4=x_{12}\end{array}$\\
			\hline
			(123),(123) & (-1,-2,1,2) & (3,2,4,1) & $\begin{array}{c}\xx_1=\xx_2=\xx_3=\xx_4\\=x_{32}^2x_{11}\Delta_{q,12,12}\end{array}$ & $\begin{array}{c}\xx_1'=x_{21}\\\xx_2'=\xx_3'=x_{33}\\\xx_4'=x_{23}\end{array}$\\
			\hline
			(123),(132) & (-1,-2,2,1) & (2,3,1,4) & $\begin{array}{c}\xx_1=\xx_2=\xx_3=\xx_4\\=x_{11}\Delta_{q,12,12}\end{array}$ & $\begin{array}{c}\xx_1'=\xx_2'=x_{33}\\\xx_3'=\xx_4'=x_{22}\end{array}$\\
			\hline
			(123),(13) & (1,2,1,-1,-2) & (3,4,2,5,1) & $\begin{array}{c}\xx_1=\xx_2\\=x_{21}^2x_{33}\Delta_{q,23,23}\\\xx_3=\xx_4=\xx_5\\=x_{33}\Delta_{q,23,23}\end{array}$ & $\begin{array}{c}\xx_1'=\xx_2'=\Delta_{q,13,12}\\\xx_3'=\xx_4'=\Delta_{q,12,12}\\\xx_5'=x_{12}\end{array}$\\
			\hline
			(132),(123) & (-2.-1,1,2) & (2,3,1,4) & $\begin{array}{c}\xx_1=\xx_2=\xx_3=\xx_4\\=x_{11}\Delta_{q,12,12}\end{array}$ & $\begin{array}{c}\xx_1'=\xx_2'=\Delta_{q,23,23}\\\xx_3'=\xx_4=x_{33}\end{array}$\\
			\hline
			(132),(13) & (1,2,1,-2,-1) & (2,3,4,1,5) & $\begin{array}{c}\xx_1=\xx_2=\xx_3\\=\xx_4=\xx_5\\=x_{33}\Delta_{q,13,12}^2\\\Delta_{q,23,23}\end{array}$ & $\begin{array}{c}\xx_1'=\xx_2'=\xx_3'=x_{21}\\\xx_4'=\Delta_{q,23,13}\\\xx_5'=\Delta_{q,12,13}\end{array}$\\
			\hline
			(13),(13) & (-1,1,-2,2,-1,1) & (3,4,1,2,5,6) & $\begin{array}{c}\xx_1=\xx_2=\xx_3\\=\xx_4=\xx_5=\xx_6\\=x_{13}x_{31}\Delta_{q,12,23}\\\Delta_{q,23,12}\end{array}$ & $\begin{array}{c}\xx_1'=\xx_2'=x_{33}\\\xx_3'=\xx_4'=x_{23}\\\xx_5'=\xx_6'=x_{32}\end{array}$\\
			\hline
		\end{tabular}
	\end{table}

	\begin{conj}
		For any semisimple Lie group $G$ and $(\ww)\in\wtw$, there exists some $\tii\in\I_{\ww}$ such that $\tirr\ttii$ is a bundle of type $(\ww)$ simple $\cqg$-modules onto $\prim\cqg_{\ww}$.
	\end{conj}

\end{document}